\documentclass[11pt]{amsart}
\usepackage{amssymb,amsmath,amsthm,graphics,epsfig}
\usepackage{amsfonts}
\usepackage{graphicx}
\usepackage{amssymb}
\usepackage{epstopdf}

\usepackage{mathrsfs}
\usepackage[colorlinks=true, pdfstartview=FitV, linkcolor=blue, citecolor=blue, urlcolor=blue]{hyperref}

\usepackage{booktabs}
\usepackage{graphicx}
\usepackage{tikz}
\usetikzlibrary{calc,angles,quotes,arrows.meta}

\usepackage{subfigure}
\usepackage{amsmath}
\usepackage{amsthm}
\usepackage{multicol}
\usepackage{cases}
\usepackage{booktabs}
\usepackage{hyperref}

\newtheorem{theorem}{Theorem}[section]
\newtheorem{lemma}[theorem]{Lemma}
\newtheorem{proposition}[theorem]{Proposition}
\newtheorem{corollary}[theorem]{Corollary}
\newtheorem{definition}[theorem]{Definition}
\newtheorem{remark}{Remark}

\newenvironment{Lemma}{\begin{lemma} \begin{sl}}{\end{sl} \end{lemma}}
\newenvironment{Proposition}{\begin{proposition}
		\begin{sl}}{\end{sl} \end{proposition}}

\newenvironment{Proof}{\removelastskip \noindent {\bf Proof~:} }
{ \hspace*{\fill} {\bf q.e.d.} \bigskip \noindent}

\DeclareMathOperator{\z}{\mathbf{z}}

\begin{document}
	
	\title{A General Framework for Relative Equilibria in the Symmetric  Full Gravitational N-Body Problem}
	\author{F. Crespo$^{1,3}$, H.R.  Dullin$^{2}$}
	
	\address{$^{1}$Aerospace Engineering Department, Embry-Riddle Aeronautical University\\ 1 Aerospace Blvd, Daytona Beach, FL 32114, USA}
\email{crespocf@erau.edu}
	\address{
		$^{2}$University of Sydney, School of Mathematics and Statistics\\ Sydney, NSW 2006, Australia}
\email{holger.dullin@sydney.edu.au}
	\address{
		$^{3}$Sydney Mathematical Research Institute\\ Sydney, NSW 2006, Australia}
		 	\date{}
	

	\begin{abstract}
We present a method to determine admissible configurations that lead to relative equilibria in the full gravitational $N$-body problem. The method exploits the $\mathrm{SO}(3)$ symmetry by working with its fundamental invariants and does not rely on restrictive assumptions about mass distributions or on truncating the gravitational potential. As a result it yields necessary conditions that apply to arbitrary rigid-body configurations.
    
This work is presented in two parts. In the first (this) part we assume every body is axisymmetric to give a clear, pedagogical exposition of the approach; in a forthcoming paper we remove that restriction and treat general triaxial bodies (the method remains valid but the number of variables increases). As an illustration, we systematically recover all known relative-equilibrium families for the two-body problem consisting of one sphere and one axisymmetric body, confirm earlier results in the literature, and derive new necessary conditions for certain configurations, including additional constraints for the arrow and non-Lagrangian types. In the special case of a sphere and an axisymmetric body, we also obtain sufficient conditions for relative equilibria under a monotonicity assumption on the potential.
\end{abstract}
	
	\maketitle
	
\section{Introduction}
The full gravitational $N$-body problem (FG$N$BP) accounts for the coupled rotational and translational dynamics of $N$ massive rigid bodies under gravitational attraction. The general gravitational potential between extended rigid bodies involves integrals over the bodies that cannot in general be solved in closed form, see \eqref{eq:Potential} below. Because of this often special models are considered, which include assumptions, hypotheses, series truncation, or treating limiting cases by tailoring a specific potential function that approximates the forces acting among the bodies. For instance, the usual $N$-body problem is the limit case when bodies are considered as point masses. In this paper we consider the case of axially symmetric bodies. It turns out that under this assumption some progress can be made for the case with a general potential. Specifically we are going to derive necessary conditions for relative equilibria in such systems.

The search for determining conditions leading to relative equilibria of $N$ bodies with arbitrary mass distribution has been studied in the literature. For instance, \cite{Moeckel2017} showed that when $N\geq3$, relative equilibria are not energy minimizers. This result generalizes the same property for the case of point masses \cite{Moeckel1990}. In \cite{CrespoTurner2022}, the authors provided a Hamiltonian formulation of the FG$N$BP for arbitrary $N$ and derived the $SO(3)$-reduced equations of motion with the aim of characterizing relative equilibria. However, the reduced equations are highly nonlinear and involve a large number of variables, making it difficult to extract general families of relative equilibria. 




In this article we study the axially symmetric full $N$-body problem. By axially symmetric we mean that every body is axially symmetric (or even spherically symmetric). This means that in addition to the usual Galilean symmetry of the problem we introduce an additional $SO(2)^N$ symmetry: rotating each body about its symmetry axis. For this problem it is natural to use a spatial frame, as explained in \cite{Dullin2004}, and the $SO(2)^N$ reduced phase space is $(T^*(\mathbb R^3 \times S^2))^N$. As a result we can write the equations of motion of  the axially symmetric full $N$-body problem in terms of $4N$ vectors in $\mathbb{R}^3$: positions $\mathbf x_i$, linear momenta $\mathbf y_i$, symmetry axes $\mathbf a_i$, angular momenta $\mathbf l_i$ where $i = 1, \dots, N$. 

Translation reduction fixes the center of mass at zero and the conserved total linear momentum $\sum \mathbf y_i $ at zero as well. Translation reduction is effected by introducing Jacobi vectors $\mathbf q_i$ and canonically conjugate momenta $\mathbf p_i$, $i = 1, \dots, N-1$ on $T^*\mathbb R^{3(N-1)}$. The spherically symmetric bodies can be treated as point masses, so they don't have axes or angular momentum vectors. The remaining $m$ axially symmetric bodies in $(T^*S^2)^m$ with coordinates $\mathbf a_j, \mathbf l_j$ where $j= 1, \dots, m$ and a Poisson structure with Casimirs $\mathbf a_j \cdot \mathbf a_j = 1$ and $ \mathbf a_j \cdot \mathbf l_j = \kappa_j$.

Notice that the $SO(3)$ symmetry that leads the conservation of total angular momentum is not reduced, and so the $6(N-1)+6m$ first order differential equations have $2m$ Casimirs and the 3 components of the total angular momentum as conserved quantities. For the simplest non-trivial case $N=2$ and $m=1$ the dynamics thus takes place on a 7-dimensional manifold.  After reduction by the remaining rotation about the axis of total angular momentum the fully reduced phase space is 6 dimensional. 

The case $N=2$ has received a great deal of attention in the literature. The first analysis of the relative equilibria for a rigid body in a central force field date back to Lagrange \cite{Lagrange1780}. Since then, many authors have studied different realizations of the FG$N$BP from the analytic \cite{Scheeres2006,BOUE2009,Scheeres2012} and geometric points of view \cite{WangB,Maciejewski2010,Crespo2024}. However, their conclusions are attached to each of their particular models, while our treatment is not. 

For $N=2$, the relative equilibria are classified as follows: Lagrangian equilibria, all the bodies are moving in the same plane. Non-Lagrangian, each body in a different plane. They were discovered and rediscovered along the literature \cite{Abolenaga1979,Barkin1985,WangA}, as explained in a historical discussion in \cite{Maciejewski1995}. Among the Lagrangian type, Kinoshita used the MacCullagh's \cite{MacCullagh1840} approximation of the potential \cite{Kinoshita1970} and found three types of equilibria: spoke, arrow, and float. Later on, non-Lagrangian equilibria were reported in \cite{WangA,Maciejewski1995,Maciejewski2010}. Additionally, Scheeres studied piled-up equilibria, which occur when the bodies touch each other \cite{Scheeres2012}. All this configurations are considered in our analysis. Moreover, we specify additional necessary conditions not reported before for the arrow and non-Lagrangian types to exist.

In the applications presented in Section~\ref{sec:SphereAxSyBodyAnalysis}, we focus on the case of a sphere and an axisymmetric body. We confirm earlier results in the literature \cite{Kinoshita1970,Maciejewski2010} and derive new necessary conditions for certain configurations, including additional constraints for the arrow and non-Lagrangian types (see Cases~1 and 2 in Section~\ref{sec:SphereAxSyBodyAnalysis}). Finally, for the special case of a sphere interacting with an axisymmetric body, Section~\ref{sec:Sufficient} derives sufficient conditions for the existence of relative equilibria under a monotonicity assumption on the potential.

\subsection{Notation and conventions}
We will consider the following notation and conventions throughout the paper. As a general agreement, we employ the same symbol for denoting vectors and their corresponding norms with bold and regular styles, respectively. Additionally, we employ the usual identification of 3-dimensional real vectors and anti-symmetric matrices through the mapping
\begin{equation}\nonumber
		\mathbf x\in\mathbb R^3 \longrightarrow \widehat{\mathbf x}\in\mathfrak{so}(3), \quad \mathbf x=\left(
		\begin{array}{c}
			x \\
			y \\
			z \\
		\end{array}
		\right),\; \widehat{\mathbf x}= \left(
		\begin{array}{ccc}
			0 & -z & y \\
			z & 0 & -x \\
			-y & x & 0 \\
		\end{array}
		\right),
	\end{equation}
where $\widehat{\mathbf x}$ satisfies the usual relation with the vector product $\widehat{\mathbf x}\mathbf y=\mathbf x\widehat{\mathbf y}=\mathbf x\times \mathbf y$. Moreover, we consider the following basic rotation matrices associated with the Euler chart of $SO(3)$
\begin{equation}\nonumber
	R_1(\alpha)=\left(
	\begin{array}{ccc}
		1 & 0 & 0 \\
		0 & \cos \alpha & \sin\alpha \\
		0 & -\sin\alpha & \cos\alpha \\
	\end{array}
	\right),\qquad R_3(\alpha)=\left(
	\begin{array}{ccc}
		\cos \alpha & \sin\alpha&0 \\
		-\sin\alpha & \cos\alpha &0\\
		 0 & 0 &1\\
	\end{array}
	\right).
\end{equation}
They allow us to relate coordinates of a given vector $\mathbf x\in\mathbb R^3$ in the spatial ($\left[\mathbf x\right]_E$) and body ($\left[\mathbf x\right]_E$) frames. Precisely, we have the following spatial-to-body transformation
\begin{equation}\nonumber
	\left[\mathbf x\right]_B=R_3\left( \psi\right)R_1\left( \theta\right) R_3\left( \phi\right) \left[\mathbf x\right]_E.
\end{equation}
For short, we will refer to the spatial to body matrix as $R=R_3\left( \psi\right)R_1\left( \theta\right) R_3\left( \phi\right) $, and the body to the spatial matrix as $A=R^T$. In our development, we do not consider a specific chart for $SO(3)$ but redundant variables with a Poisson structure. Therefore, we will consider the matrix $A$ as a collection of vectors
\begin{equation}\nonumber
	A=\left[\mathbf a_1,\mathbf a_2,\mathbf a_3\right]
\end{equation}
where $\mathbf a_1,\mathbf a_2,\mathbf a_3$ are columns vectors, being $\mathbf a_i=\left[\mathbf b_i\right]_E $ the coordinates of the body $\mathbf b_i$ axis in the spatial frame.

\section{Equations of Motion}
Most part of this paper is devoted to the case of bodies with axial symmetry. However, we start deriving the equations of motion at its maximum generality and then we will specialize to the axial symmetric case.

\subsection{Equations of motion for $N$ triaxial bodies}
In this part, we provide the equations of motion for the full gravitational $N$-body problem following \cite{CrespoTurner2022} and \cite{Dullin2004}. In \cite{CrespoTurner2022}, the authors set this problem in the Hamiltonian formalism through a suitable Poisson structure \cite{Vaisman1994}. The approach in \cite{Dullin2004} differs from \cite{CrespoTurner2022} in the number of dimensions involved and in the fact that the equations are expressed in the spatial frame. Here, thinking in applications to axial-symmetric bodies, we also chose variables in the spatial frame.  

We assume that we have $N$ finite bodies with masses $(m_1,\ldots,m_N)$, interacting gravitationally. Hence, the phase space is $T^*\mathbb R^{3N}\times T^*SO(3)^N$. Precisely, we employ the following notation for the coordinates and conjugate momenta
\begin{equation}
\label{eq:Variables1}
\mathbf q=(\mathbf q_1,\dots,\mathbf q_N),\quad \mathbf p=(\mathbf p_1,\dots,\mathbf p_N),\quad \pmb{l}=({\pmb{l}}_1,\dots,{\pmb{l}}_N),\quad A=(A_1,\dots,A_N),
\end{equation}
where $A_i=\left[\mathbf a_{i1},\mathbf a_{i2},\mathbf a_{i3} \right]\in SO(3)$ are the attitude matrices with columns $ \mathbf a_{ij}\in \mathbb R^3 $, and ${\pmb{l}}_i\in\mathbb{R}^3$ are the spatial angular momenta. Moreover, $\mathbf q_i\in\mathbb{R}^3,\mathbf p_i\in\mathbb{R}^3$ are the Jacobi vectors and conjugate momenta. Now, we consider the inertia tensor for each body $\mathbb I_i=Diag(\Theta_{i1},\Theta_{i2},\Theta_{i3})$ in the body frame, and employ the notation $\mathbf z_1=\left[\mathbf q,\mathbf p,{\pmb{l}},A\right] $ to express the Hamiltonian of the system whose value is the total energy as
\begin{equation}
\label{eq:Hamiltonian1}
	\mathcal{H}_1(\mathbf z_1)=\dfrac{1}{2}\left[ \sum_{i=1}^{N}\frac{\Vert \mathbf p_i\Vert }{m_i}+\sum_{i=1}^N\left\langle A_i^T{\pmb{l}}_i,\mathbb I_i^{-1}A_i^T	{\pmb{l}}_i\right\rangle \right] +V(\mathbf q,A),
\end{equation}
with potential functions given by
\begin{gather}\label{eq:Potential}
	\begin{aligned}
		V(\mathbf q,A)&=-\mathcal{ G}\sum_{\stackrel{i=1}{i<j}}^{N}  \int_{\mathcal{B}_i}\int_{\mathcal{B}_j}\dfrac{dm(\mathbf{X}_i)dm(\mathbf{X}_j)}{\Vert\mathbf q_i-\mathbf q_j+A_i\cdot\mathbf{X}_i-A_j\cdot \mathbf{X}_j\Vert},
	\end{aligned}
\end{gather}
where $\mathcal{B}_k$ correspond with the $k$-th body, and  $\mathcal{ G}$ is the universal gravitational constant. Then,  the equations of motion are 
\begin{gather}
\label{eq:EqMotion1}
	\begin{aligned}
		&\dot {\mathbf q}_i=\dfrac{\mathbf p_i}{\nu_i},\qquad \dot {\mathbf p}_i=-\dfrac{\partial V}{\partial \mathbf q_i}(\mathbf q,A),\\
		& \dot {\mathbf a}_{ij}=\dfrac{\partial \mathcal{H}_1}{\partial {\pmb{l}}_i}\times \mathbf a_{ij},\qquad \dot{\pmb{l}}_i=-\dfrac{\partial \mathcal{H}_1}{\partial {\pmb{l}}_i}\times{\pmb{l}}_i-\sum_{j=1}^3\dfrac{\partial \mathcal{H}_1}{\partial \mathbf a_{ij}}\times \mathbf a_{ij}.
	\end{aligned}
\end{gather}
Denoting $\mathcal M_1:=T^*\mathbb R^{3N}\times T^*SO(3)^N$, the above equations are Hamiltonian in the Poisson manifold $(\mathcal M_1,\left\lbrace \,,\,\right\rbrace_1)$. That is to say, we may express equations \eqref{eq:EqMotion1} in the following compact form 
$$\dot{\mathbf z}_1=\lbrace \mathbf z_1,\mathcal H_1\rbrace_1,$$ 
where the Poisson structure is given as
\begin{equation}
\label{eq:PoissonStructure1}
	\left\lbrace f,g\right\rbrace_1=\nabla f^T\cdot J_1(\mathbf z_1)\cdot \nabla g 
\end{equation}
with the following structure matrix
\begin{equation}\nonumber
		J_1(\mathbf z)=\left[
		\begin{array}{cccccc}
			J_{S}  		   &0		 &\cdots&0 \\
			0		    		   &J_1({\pmb{l}}_1,A_1 )&0&\vdots\\
			\vdots    		 		   &     0			 &\ddots&0\\
			0				 		   &	\cdots 			 &		0	& J_N({\pmb{l}}_N,A_N)\\
		\end{array}
		\right],\quad
		J_{i}({\pmb{l}}_i,A_i)=\left[
\begin{array}{cccc}
	-\widehat{{\pmb{l}}}_i&-\widehat{\mathbf a}_{i1}&-\widehat{\mathbf a}_{i2}&-\widehat{\mathbf a}_{i3}\\
	-\widehat{\mathbf a}_{i1}& 0 & 0 & 0  \\
	-\widehat{\mathbf a}_{i2}& 0 & 0 & 0  \\
	-\widehat{\mathbf a}_{i3}& 0 & 0 & 0  \\
\end{array}
\right]
\end{equation}
where $J_{S} $ is the standard symplectic matrix of order $2N$. This Poisson structure is endowed with the following Casimirs
	\begin{gather}
    \label{eq:Casimirs}
		\begin{aligned}
			&c_{1i}=\left\langle {\mathbf a}_{i1},{\mathbf a}_{i1}\right\rangle,\ \ \ \  c_{2i}=\left\langle {\mathbf a}_{i1},{\mathbf a}_{i2}\right\rangle,\ \ \ \  c_{3i}=\left\langle{\mathbf a}_{i1},{\mathbf a}_{i3} \right\rangle,\\ 
			&c_{4i}=\left\langle {\mathbf a}_{i2},{\mathbf a}_{i2}\right\rangle,\ \ \ \ \ c_{5i}=\left\langle {\mathbf a}_{i2},{\mathbf a}_{i3}\right\rangle,\ \ \ \  c_{6i}=\left\langle{\mathbf a}_{i3},{\mathbf a}_{i3} \right\rangle.
		\end{aligned}
	\end{gather}

\begin{remark}
    From the definition of the Hamiltonian function $\mathcal H_1$ given by \eqref{eq:Hamiltonian1} and \eqref{eq:Potential}, it is clear that $\mathcal H_1$ is $O(3)$-invariant. This fact will be fundamental for the subsequent development. Surprisingly, the Hamiltonian equations of motion are not $O(3)$-invariant, but $SO(3)$-invariant. The $O(3)$-invariance fails for equations $\mathbf a_{ij}$ and $\pmb l_j$ due to the non-standard Poisson structure associated with the rotational part.
\end{remark}

\subsection{The case of axially symmetric and spherical bodies}
Here we restrict to the case in which we have $m$ axially symmetric bodies and $N-m$ spheres. As a consequence of this assumption, the number of degrees of freedom decreases considerably. Precisely, the attitude matrices and angular momenta associated with spherical bodies are no longer needed. Without loss of generality, we may assume that the spherical bodies correspond with the last $N-m$ indices. Therefore, we represent the phase space with fewer variables, which equals to replace $N$ by $m$ in the above formulae \eqref{eq:Variables1}, \eqref{eq:Hamiltonian1}, \eqref{eq:EqMotion1} and \eqref{eq:PoissonStructure1}. Moreover, we have an extra symmetry given by the group $SO(2)^m$ acting as we describe through $\Psi$
\begin{eqnarray}
\nonumber		\Psi: SO(2)^m\times M_1&\to& M_1\\
\nonumber		(g,\mathbf z_1)&\mapsto&\Psi(g,\mathbf z_1)=g*\mathbf z_1
\end{eqnarray}
where $g\in SO(2)^m$ is given by $g=(g_1,\dots,g_m)$, with $g_i\in SO(2)$, and the operation $(*)$ is defined as
\begin{gather}\nonumber
	\begin{aligned}
		&g*\mathbf z_1=\left[\mathbf q,\mathbf p,\pmb{l},g*A\right], \qquad g*A=\left[ g_i\cdot A_i\right]_{i=1,\dots,m}.
	\end{aligned}
\end{gather}
More in detail, we have that $g_i=R_{\mathbf a_{i3}}(\alpha_i)$ is the rotation matrix with axis $\mathbf a_{i3}$ and arbitrary angle $\alpha_i$. The invariants associated to the action $\Psi$ are $\mathbf q$, $\mathbf p$, $\pmb{l}$ and the unit vectors $\mathbf a_{i3}$ for $i=1,\ldots,m$, which are the third column of each matrix $A_i$. 

In what follows, we will drop the index 3 when we refer to the body axis $\mathbf a_{i3}$, and we introduce the notation $\mathbf a=(\mathbf a_1,\ldots,\mathbf a_m)$ for the collection of all axes. Therefore, the $SO(2)^m$-reduced phase space is given by $\mathcal M_2:=T^*\mathbb R^{3N}\times T^*(\mathbb S^2)^m$, and described with the variables $\mathbf z_2=\left[\mathbf q,\mathbf p,\pmb{l},\mathbf a\right]$. This manifold is endowed with the reduced Poisson bracket $\{\,,\,\}_2$ given by the following Poisson structure matrix
\begin{equation}
\label{eq:PoissonStructure2}
J_2(\mathbf z_2)=\left[
\begin{array}{cccccc}
	J_{S}  		   &0		 &\cdots&0 \\
	0		    		   &J_2(\pmb{l}_1,\mathbf a_1 )&0&\vdots\\
	\vdots    		 		   &     0			 &\ddots&0\\
	0				 		   &	\cdots 			 &		0	& J_2(\pmb{l}_m,\mathbf a_{m})\\
\end{array}
\right], \quad J_2(\pmb{l}_j,\mathbf  a_j)=\left[\begin{array}{cc}
	-\widehat{\pmb{l}}_j&-\widehat {\mathbf a}_{j}\\
	-\widehat {\mathbf a}_{j}&0
\end{array}\right],\; j=1,\dots,m,
\end{equation}
with Casimir functions $\left\langle {\mathbf a}_{i},{\mathbf a}_{i}\right\rangle = c_{i}$
and $\left\langle {\mathbf l}_{i},{\mathbf a}_{i}\right\rangle = L_{i3}$ 
where the subscript $i3$ indicates that this is the 3rd component of the angular momentum vector of body $i$ in the body frame.

As observed in \cite{Dullin2004} with the additional $SO(2)^m$ axial symmetry of $m$ bodies the rotational kinetic energy can be written in the space fixed frame and depends only on the length of the angular momentum of the body. We are briefly going to recall the argument for a single body, thus dropping the subscript $j$. Denote by $\mathbf L =  A^t \mathbf l$ the angular momentum vector in the body frame, where $A \in SO(3)$ is the rotation matrix that connects the spatial and the body frame, and $\mathbf l$ is the angular momentum in the spatial frame. Furthermore, denote the tensor of inertia of the body by $\mathbb I=Diag(\Theta_{1},\Theta_{1},\Theta_{3})$. Then we have
\[
   K = \frac12 \mathbf L^t \mathbb I^{-1} \mathbf L = \frac12 \mathbf l^t A \mathbb I^{-1} A^t \mathbf l ,
\]
so that the kinetic energy depends on $A$, and that is why usually the body frame is preferred. However, with axial symmetry we can write 
\[
\mathbb I^{-1} = \frac{1}{\Theta_1} I + \left( \frac{1}{\Theta_3} - \frac{1}{\Theta_1} \right) \mathbf e_z \mathbf e_z^t
\]
where $\mathbf e_z$ is the symmetry axis of the body in the body frame. Now define $\pmb  a = A \mathbf e_z$ as the symmetry axis of the body in the spatial frame, and the kinetic energy can be written as
\[
 K = \frac{1}{2 \Theta_1} \mathbf l^2 + \frac12  \left( \frac{1}{\Theta_3} - \frac{1}{\Theta_1} \right)  \langle \pmb a, \mathbf l\rangle^2 \,.
\]
Furthermore, the projection $\langle \mathbf L, \mathbf e_z \rangle = \langle \mathbf l, \pmb a \rangle$ is a constant of motion, so that up to a constant the kinetic energy is proportional to $\mathbf l^2$. The vector $\pmb a$ has length one, because $\mathbf e_z$ does, and the two vectors $(\pmb a, \mathbf l)$ come with a Lie-Poisson structure as derived in \cite{Dullin2004}. The constant of motion $\pmb a^2$ and $ \langle \mathbf l, \pmb a \rangle$ are Casimirs of this Poisson structure. Fixing the Casimirs to $1$ and $0$, respectively, we obtain a system on $T^* S^2$. The argument can be repeated for each body, and the kinetic energy becomes a sum depending on $\mathbf l_j^2$, $j = 1, \dots, m$.

Considering the above treatment of the kinetic energy, and keeping in mind that the potential function $V$ given in \eqref{eq:Potential} is invariant under the action $\Psi$, we have that the reduced Hamiltonian is given by the following function
\begin{equation}
\label{eq:Hamiltonian2}
	\mathcal{H}_2(\mathbf z_2)=\dfrac{1}{2}\left[ \sum_{i=1}^N\dfrac{\Vert \mathbf p_i\Vert^2}{m_i}+\sum_{j=1}^m\dfrac{\Vert \pmb{l}_j\Vert ^2}{\Theta_{j1}}\right] +V(\mathbf q,\mathbf a)+c
\end{equation}
where, by a slight abuse of notation, we denote the new potential function by the same symbol, and
$c=\sum_{j=1}^m\left(\dfrac{1}{\Theta_{j3}} -\dfrac{1}{\Theta_{j1}} \right)\langle \pmb{l}_j, \mathbf a_j\rangle^2$
is a constant term that does not affect the equations of motion $\dot {\mathbf z}_2=\{\mathbf z_2,\mathcal{H}_2\}_2$, or explicitly
\begin{gather}
\label{eq:EqMotion2}
	\begin{aligned}
		&\dot {\mathbf q}_i=\dfrac{\mathbf p_i}{\mu_i},\qquad \dot {\mathbf p}_i=-\dfrac{\partial V}{\partial \mathbf q_i}, \quad i=1, \dots, N\\
		& \dot {\mathbf a}_{j}=\dfrac{1}{\Theta_{j1}}\pmb {l}_j\times \mathbf a_{j},\qquad \dot{\pmb {l}}_j=\dfrac{\partial V}{\partial \mathbf a_j}\times \mathbf a_j, \quad j = 1, \dots, m \,.
	\end{aligned}
\end{gather}

These are the equations reduced by the body symmetries $SO(2)^m$. Notice that the $SO(3)$ symmetry of the equations has not been reduced, and the corresponding constant of motion is the total angular momentum
\[
\mathbf c = \sum_i^N \mathbf q_i \times \mathbf p_i + \sum_j^m \mathbf l_j \,.
\]

\begin{remark}
    The potential function for the case of axial symmetric bodies depends only on the variables $(\mathbf q,\mathbf a)$. More precisely, the dependence on $\mathbf{q}$ is given through the differences $\mathbf{q}_i-\mathbf{q}_j$ for $i,j\in\{1,\ldots,N\}$ .
\end{remark}

\section{Reduction of the translational symmetry}

The symmetry group of the FG$N$BP at its maximum generality is given by the translational and rotational symmetries associated with the Lie groups $\mathbb R^3$, $SO(3)$ and the Galilean boosts, see \cite{Dullin2004,CrespoTurner2022}. This section reproduces the translational symmetry following the process described in \cite{CrespoTurner2022}, where we refer to the reader for all the details. Precisely, the transition of the origin to the center of mass 
\begin{equation}
	\mathbf{\tilde{\mathbf q}}_i=\mathbf q_i-\mathbf q_{cm},\qquad \tilde{\mathbf p}_i=\mathbf p_i-\dfrac{m_i}{\textrm m}\sum_{j=1}^n m_j\dot{{\mathbf q}}_j,
\end{equation}
where $\mathbf q_{cm}$ is the centre of mass,  $m_s=m_1+\cdots +m_N$ is the total mass of the system, and we consider the notation $\tilde{\mathbf q}=(\tilde{\mathbf q}_1,\ldots,\tilde{\mathbf q}_N)$, $\tilde{\mathbf p}=(\tilde{\mathbf p}_1,\ldots,\tilde{\mathbf p}_N)$. In this set of variables, the orbital phase space restricts to a $6(N-1)$-dimensional vectorial subspace $M\subset T^*\mathbb R^{3N}$ defined by 
\begin{equation}
    \label{eq:M2Constrains}
    \sum_{i=1}^Nm_i\tilde{\mathbf q}_i=0,\qquad \sum_{i=1}^N\tilde{\mathbf p}_i=0.
\end{equation}
Thus, by this transformation, the phase space of the FG$N$BP is restricted to the manifold $\tilde{\mathcal{M}}_2={M}\times T^*(\mathbb S^2)^m$. The new Poisson structure $\{,\}_2$ associated with $\mathcal{M}_2$ is given in \cite{CrespoTurner2022}. Its Poisson matrix is obtained by replacing the standard Poisson matrix $J_S$ in \eqref{eq:PoissonStructure2} with $\mathbf J_{M} $, which is given by 
$$\mathbf J_{M} =\left[\begin{array}{cc}0 & \bar{M} \\-\bar{M} & 0\end{array}\right],\quad \bar{M}=\frac{1}{m_s}\left[\begin{array}{ccccc}\bar{m}_s-\bar{m}_1 & -\bar m_2 & \cdots & \cdots & -\bar m_N \\-\bar m_1 & \bar{m}_s-\bar m_2 & -\bar m_3 & \cdots & -\bar m_N \\\vdots & \vdots & \vdots & \vdots & \vdots \\-\bar m_1 & \cdots & \cdots & -\bar{m}_{N-1} & \bar{m}_s-\bar m_N\end{array}\right],$$
where $\bar{M}$ is a $3N\times 3N$ matrix, whose entries $\bar{m}_{x}$ denotes a $3\times 3$ matrix given by ${m}_{x} I_3$, and $I_3$ is the 3-dimensional identity matrix.

Moreover, after some algebraic manipulations and dropping constant terms, the Hamiltonian in the new variables retains its original shape 
\begin{equation}
\label{eq:Hamiltonian2Tilde}
	\mathcal{H}_2(\tilde{\mathbf z}_2)=\dfrac{1}{2}\left[ \sum_{i=1}^N\dfrac{\Vert \tilde{\mathbf p}_i\Vert^2}{m_i}+\sum_{j=1}^m\dfrac{\Vert \pmb{l}_j\Vert ^2}{\Theta_{j1}}\right] +V(\tilde{\mathbf q},\mathbf a)+c,
\end{equation}
where $\tilde{\mathbf z}_2=\left[\tilde{\mathbf q},\tilde{\mathbf p},\pmb{l},\mathbf a\right]$, and the expression of the potential function $V$ defined in \eqref{eq:Potential} remains unchanged for the variables referred to the center of mass. Thus, the equations of motion become
\begin{gather}
\begin{aligned}
\label{eq:EquationOfMotionBoosts}
	 &\dot{\tilde{\mathbf q}}_i=\frac{\tilde{\mathbf p}_i}{m_i},\quad \dot{\tilde{\mathbf p}}_i =\dfrac{1}{m_s}\left[(m_i-m_s)\dfrac{\partial V}{\partial \tilde{\mathbf q}_i} +m_i \sum_{i\neq j}\dfrac{\partial V}{\partial \tilde{\mathbf q}_j}\right], \qquad i=1, \dots, N\\
		& \dot {\mathbf a}_{j}=\dfrac{1}{\Theta_{j1}}\pmb {l}_j\times \mathbf a_{j},\quad \dot{\pmb {l}}_j=\dfrac{\partial V}{\partial \mathbf a_j}\times \mathbf a_j, \qquad j = 1, \dots, m \,.
\end{aligned}
\end{gather}
Notice that the variables $(\tilde{\mathbf q},\tilde{\mathbf p})$ are redundant since \eqref{eq:M2Constrains} must hold. Therefore, the transition of the origin to the centre of mass has decreased the degrees of freedom by three. In this situation, barycentric, or Jacobi coordinates as well as relative coordinates are usually employed. For our purposes, relative coordinates are more suitable since Jacobi coordinates bring a highly complex expression for the Hamiltonian, see \cite{Meyer} page 199.
\begin{equation}
	{\mathbf q}_{ij}=\tilde{\mathbf q}_i-\tilde{\mathbf q}_j\qquad \mathbf p_{ij}=\frac{1}{N}(\tilde{\mathbf p}_i-\tilde{\mathbf p}_{j}).\nonumber
\end{equation}
Following \cite{CrespoTurner2022}, this process is executed by means of the following linear map
$$\tilde\Phi:T^*\mathbb{R}^{3N}\to T^*\mathbb R^{\frac{3N(N-1)}{2}},\quad (\tilde{\mathbf q},\tilde{\mathbf p})\to \tilde\Phi(\tilde{\mathbf q},\tilde{\mathbf p})=({\mathbf Q},{\mathbf P}),$$
where ${\mathbf Q}:=({\mathbf q}_{12},\ldots,{\mathbf q}_{ij},\ldots,{\mathbf q}_{N-1\,N})$ and ${\mathbf P}:=({\mathbf p}_{12},\ldots,{\mathbf p}_{ij},\ldots,{\mathbf p}_{N-1\,N})$, with $1\leq i<j\leq N$. Now, we consider $\hat{\Phi}=\tilde\Phi_{|{M}}$, which is the restriction of $\tilde\Phi$ to the subspace ${M}$. Note that the image of this map ${T}:=\hat \Phi({M})$ is a $6(N-1)$-dimensional subspace in $T^*\mathbb R^{\frac{3N(N-1)}{2}}$, which is given by the following linear restrictions
\begin{equation}
	 T=\left\lbrace ({\mathbf Q},{\mathbf P})\in T^*\mathbb R^{\frac{3N(N-1)}{2}}: {\mathbf q}_{ij}+{\mathbf q}_{jk}-{\mathbf q}_{ik}=0\,\wedge \,  {\mathbf p}_{ij}+{\mathbf p}_{jk}-{\mathbf p}_{ik}=0  \right\rbrace.
\end{equation}
The map $\hat\Phi$ is a linear isomorphism between ${M}$ and ${T}$. Adding the identity on $T^*SO(3)^N$, $\hat\Phi$ may be extended to a transformation $\Phi$ between $\tilde{\mathcal{M}}_2$ and $\mathcal{M}_3= T\times T^*(\mathbb S^2)^m$. Proceeding in the same way as in the previous section, we compute the new Poisson bracket $\{,\}_3$ defined in $C^\infty(T^*\mathbb R^{\frac{3N(N-1)}{2}}\times \mathbb R^{6N})$ as
	\begin{equation}
	\label{eq:J3}
		\mathbf J_3(\z)=\left[
\begin{array}{cccccc}
	J_{T}  		   &0		 &\cdots&0 \\
	0		    		   &J_2(\pmb{l}_1,\mathbf a_1 )&0&\vdots\\
	\vdots    		 		   &     0			 &\ddots&0\\
	0				 		   &	\cdots 			 &		0	& J_2(\pmb{l}_m,\mathbf a_{m})\\
\end{array}
\right],
	\end{equation}
where ${\mathbf z}=\left[{\mathbf Q},{\mathbf P},\pmb{l},\mathbf a\right]$, and  $\mathbf J_{T} $ is a $3N(N-1)\times 3N(N-1)$ matrix given by 
\begin{equation}
\label{eq:MatrixJT}
\mathbf J_{T} =\left[\begin{array}{cc}0 & \bar{T}  \\-\bar{T} & 0 \end{array}\right].
\end{equation}
The block $\bar{T}$ is a $(\frac{3N(N-1)}{2})\times (\frac{3N(N-1)}{2})$ symmetric matrix. In order to specify the components of $\bar{T}$, we organize its entries by $(3\times 3)$-diagonal matrices denoted by $\bar{t}_{(ij),(lk)}$, which are obtained from the bracket of the components of ${\mathbf q}_{ij}$ and ${\mathbf p}_{lk}$. Thus, we have $\bar{t}_{(ij),(lk)} ={t}_{(ij),(lk)} I_3$, where $I_3$ is the 3-dimensional identity matrix and the scalars ${t}_{(ij),(lk)}$ satisfy 
$${t}_{(ij),(ij)}=\frac{2}{N},\quad {t}_{(ij),(jk)}= {t}_{(ij),(li)}=\frac{-1}{N},\quad {t}_{(ij),(lj)}={t}_{(ij),(ik)}=\frac{1}{N},$$
and ${t}_{(ij),(lk)}=0$ otherwise. For instance, for the case $N=4$ we have 
$$ \bar T=\frac{1}{N}	\left[
\begin{array}{rrrrrr}
 2  I_3 &I_3 & I_3 &- I_3 &- I_3 &0_3 \\
 I_3& 2  I_3 & I_3& I_3&0_3&- I_3\\
 I_3& I_3& 2  I_3 &0_3& I_3& I_3\\
- I_3& I_3&0_3& 2  I_3 & I_3&- I_3\\
- I_3&0_3& I_3& I_3& 2  I_3 & I_3\\
0_3&- I_3& I_3&- I_3& I_3& 2  I_3 \\
\end{array}
\right].$$
A straightforward computation shows that the Poisson structure $\{,\}_3$ is endowed with the Casimirs $\left\langle {\mathbf a}_{i},{\mathbf a}_{i}\right\rangle = c_{i}$
and $\left\langle {\mathbf l}_{i},{\mathbf a}_{i}\right\rangle = L_{i3}$ together with
\begin{gather}
\begin{aligned}
\label{eq:CasimirsTranslational}
c_{q(ij),(jk)}= {\mathbf q}_{ij}+{\mathbf q}_{jk}-{\mathbf q}_{ik},\quad  c_{p(ij),(jk)}= {\mathbf p}_{ij}+{\mathbf p}_{jk}-{\mathbf p}_{ik}.
\end{aligned}
\end{gather}
The manifold $\mathcal{M}_3$ is obtained as the $\mathbb{Z}_2^N$-reduced symplectic leaf corresponding with fixed values for $L_{i3}$,  $c_{i}=1$, and $c_{q(ij),(jk)}= c_{p(ij),(jk)}= 0$, and the expression of the Hamiltonian in these variables is easily found as  $\mathcal{ H}_3(\mathbf{z}):=T_3(\mathbf{z})+V(\mathbf{z})$, where
\begin{gather}
\begin{aligned}
	\label{eq:T3U3}
	T_3(\mathbf{z})=&\frac{1}{2}\left[\sum_{i=1}^N\frac{\left\| \sum_{j\neq i} {\mathbf p}_{ij}\right\|  ^2}{m_i}+\sum_{j=1}^m\dfrac{\Vert \pmb{l}_j\Vert ^2}{\Theta_{j1}}  \right],\\
	V(\mathbf z)=&-\mathcal{ G}\sum_{\stackrel{i=1}{i<j}}^{N}  \int_{\mathcal{B}_i}\int_{\mathcal{B}_j}\dfrac{dm(\mathbf{X}_i)dm(\mathbf{X}_j)}{\Vert\mathbf q_{ij}+A_i\cdot\mathbf{X}_i-A_j\cdot \mathbf{X}_j\Vert}.
\end{aligned}
\end{gather}

\begin{remark}
    Notice that the expression of the potential remains unchanged throughout the development because it is written in relative coordinates from the outset. This feature represents a significant advantage over Jacobi coordinates, which produce very complex expressions for the potential function \cite{Meyer}.
\end{remark}

Finally, the Hamilton-Poisson equations of motion are explicitly given as 
\begin{gather}
	\begin{aligned}
		&\dot{\mathbf p}_{ij}=-\frac{\partial V}{\partial {\mathbf q}_{ij}},\qquad \dot{\mathbf q}_{ij}=\frac{1}{m_i}\sum_{k\neq i}\mathbf p_{ik}+\frac{1}{m_j}\sum_{k\neq j}\mathbf p_{kj},\\
		& \dot {\mathbf a}_{j}=\dfrac{1}{\Theta_{j1}}\pmb {l}_j\times \mathbf a_{j},\quad \dot{\pmb {l}}_j=\dfrac{\partial V}{\partial \mathbf a_j}\times \mathbf a_j, \qquad j = 1, \dots, m \,.
	\end{aligned}
\label{eq:EquationOfMotionReduced}
\end{gather}

Instead of performing a further symmetry reduction of the equations, in the following section we seek relative equilibria using the equations in their current form. Note also that, unlike the usual formulation in rigid body dynamics, the angular momentum and the axis of each body are described here by vectors in the spatial frame, rather than in a frame rotating with the body.

\section{Relative equilibria. Properties and equations}
\label{sec:RE}
We study the relative equilibria of system \eqref{eq:EquationOfMotionReduced} associated with the remaining $SO(3)$ symmetry. In this section we first discuss several general properties and derive the equations characterizing relative equilibria in a convenient form.

Before determining the conditions for relative equilibria, we recall some key facts and introduce notation that will be used throughout the analysis.

In particular, we consider the constrains imposed by the translational Casimirs \eqref{eq:CasimirsTranslational}
\begin{equation}
    \label{eq:CasimirsTranslational0}
    \mathbf q_{ij}+\mathbf q_{jk}-\mathbf q_{ik}=0,\qquad
\mathbf p_{ij}+\mathbf p_{jk}-\mathbf p_{ik}=0.
\end{equation}

They play a crucial practical role in reducing the number of independent variables. These linear relations show that all relative position and momentum vectors are algebraically determined by the $(N-1)$ vectors that connect body~1 to the others. Concretely one convenient choice of a \emph{fundamental} set is
\[
\bar{\mathbf Q}:=(\mathbf q_{12},\mathbf q_{13},\dots,\mathbf q_{1N}),\qquad
\bar{\mathbf P}:=(\mathbf p_{12},\mathbf p_{13},\dots,\mathbf p_{1N}),
\]
from which every other relative vector is recovered by the linear identities
\begin{equation}
\label{eq:recover}
\mathbf q_{ij}=\mathbf q_{1j}-\mathbf q_{1i},\qquad
\mathbf p_{ij}=\mathbf p_{1j}-\mathbf p_{1i},
\end{equation}
which are equivalent to the Casimir relations above. Depending on the calculation it is often advantageous to switch between the two descriptions: the potential $V$ is most simply written using the full set ${\mathbf Q}$, whereas many equilibrium conditions and linear-algebra manipulations are more compact when written in terms of the fundamental sets $\bar{\mathbf Q},\bar{\mathbf P}$. We will therefore adopt the notation \(\bar{\mathbf Q},\bar{\mathbf P}\) whenever we reduce to the minimal set of translational variables, and otherwise retain the full collection \(\mathbf Q,{\mathbf P}\) when that proves clearer for expressing $V$ or its derivatives.
\medskip

\subsection{General properties}
Relative equilibria are simultaneously solutions of the reduced system \eqref{eq:EquationOfMotionReduced} and orbits of the $SO(3)$-action. Therefore, assuming without loss of generality that $\pmb{\omega}=\omega \mathbf e_z$, we obtain that relative equilibria have the following form:
\begin{gather}
\label{eq:Ansatz}
	\begin{aligned}
		& \mathbf q_{ij}(t)=e^{\widehat{\pmb \omega}t}\mathbf q_{ij0},\qquad \mathbf p_{ij}(t)=e^{\widehat{\pmb \omega}t}\mathbf p_{ij0},\qquad \mathbf a_k(t)=e^{\widehat{\pmb \omega}t}\mathbf a_{k0},\qquad \pmb{l}_k(t)=e^{\widehat{\pmb \omega}t}\pmb {l}_{k0} \,.
	\end{aligned}
\end{gather}
Hence, the derivatives of $\mathbf q_{ij}$, $\mathbf p_{ij}$, $\mathbf a_k$ and $\pmb{l}_k$ all satisfy the differential equation
\begin{equation}
    \label{eq:Derivative}
    \dot{\pmb{x}}_\sigma=\pmb{\omega} \times\pmb{x}_\sigma.
\end{equation}

In this situation, the body-to-spatial coordinates matrix $A_i$ is composed of three column vectors: $\mathbf a_i$ and two arbitrary orthonormal vectors lying in the plane orthogonal to $\mathbf a_i$. Hence, the following proposition holds.

\begin{proposition} 
\label{propo:RE}
In a relative equilibria configuration, the $i$-th body $\pmb\omega_i=\pmb\omega$ and $\pmb\Omega_i=\omega\,\mathbf a_{i0}^T$, where $\mathbf a_{i0}^T$ denotes the third column of the matrix $A_{i0}^T$, and $\pmb\omega_i$ and $\pmb\Omega_i$ denote the angular velocity of the $i$-body in the spatial and body frames, respectively. 
\end{proposition}
\begin{proof}
We can compute the angular velocity as follows
$$\widehat{\pmb\Omega}_i=A_i^T \dot A_i=(A_{i0}^T \,e^{-\widehat{\pmb \omega}t})(\widehat{\pmb \omega} \,e^{\widehat{\pmb \omega}t} A_{i0})=A_{i0}^T \,\widehat{\pmb \omega} \,A_{i0}=\widehat{A_{i0}^T \,{\pmb \omega}}.$$
Hence, we have that $\pmb\Omega_i=A_{i0}^T\, \pmb \omega=\omega\,\mathbf a_{i0}^T$. 
\end{proof}

Employing the characterization of relative equilibria \eqref{eq:Ansatz}, next we will seek necessary conditions leading to relative equilibria solutions. In \cite{CrespoTurner2022}, the authors reduced the $SO(3)$ rotational symmetry and obtained the reduced equations of motion. However, these equations involved high complexity and were not very fruitful in characterizing conditions for relative equilibria. Here, we employ a different methodology studying conditions for \eqref{eq:Ansatz} to satisfy system \eqref{eq:EqMotion2}. Previous to this task, we introduce convenient notation and present several general properties associated with the relative equilibria.

We denote the initial conditions associated with a relative equilibria as $\mathbf z_{0}=(\mathbf Q_{0},\mathbf P_{0},\pmb {l}_{0},\mathbf a_{0})$. However, for the sake of a simpler notation, we will drop the subindex $0$ in the initial conditions components 
$$(\mathbf q_{ij0},\mathbf p_{ij0},\pmb {l}_{k0},\mathbf a_{k0})\to(\mathbf q_{ij},\mathbf p_{ij},\pmb {l}_{k},\mathbf a_{k}),$$
and we employ the following notation $\mathbf q_{ij}=(x_{ij},y_{ij},z_{ij})$ and $\mathbf a_{k}=(u_k,v_k,w_k)$.

Now, by plugging in the ansatz \eqref{eq:Ansatz} into the reduced system \eqref{eq:EquationOfMotionReduced}, we will obtain relations among the initial conditions determining conditions for the existence of relative equilibria.

\begin{theorem}
\label{theorem:Cond}
Let $\mathbf z_{0}=(\mathbf Q_{0},\mathbf P_{0},\pmb {l}_{0},\mathbf a_{0})$ be an initial condition of system \eqref{eq:EquationOfMotionReduced} with time evolution given by \eqref{eq:Ansatz}. The following relations for each component of $\mathbf z_{0}$ determine the relative equilibria conditions.
\begin{description}
  \item[(i) momenta determined by coordinates] 
   \begin{gather}
\begin{aligned}
\label{eq:REmomenta}
	{\mathbf p}_{ij}&=\pmb\omega\times\left[\dfrac{m_i}{N\,m_s}\sum_{k\neq i}{m_k\,\mathbf q}_{ik}+\dfrac{m_j}{N\,m_s}\sum_{k\neq j}{m_k\,\mathbf q}_{kj}\right],\\
	\pmb{l}_j&=\Theta_{j1}\pmb\omega+k_jw_j\mathbf a_j.
\end{aligned}
\end{gather}
where the factor multiplying $\mathbf{a}_j$ is either  $k_j=\omega \left( \Theta_{j3}-{\Theta_{j1}}\right)$ or $k_j=0$.
  \item[(ii) equations for the coordinates]
  \begin{gather}
\begin{aligned}
\label{eq:RE}
	&\dfrac{1}{N}\,\widehat{\pmb\omega}^2\left[\dfrac{m_i}{m_s}\sum_{k\neq i}{m_k\,\mathbf q}_{ik}+\dfrac{m_j}{m_s}\sum_{k\neq j}{m_k\,\mathbf q}_{kj}\right]=-\dfrac{\partial V}{\partial \mathbf q_{ij}}, \\
	&\mu_j\mathbf a_j=k_jw_j\, \pmb\omega-\frac{\partial V}{\partial \mathbf a_j},\qquad 
\end{aligned}
\end{gather}
\end{description}
where $i=1,\ldots,N-1$, $j=1,\ldots,m$, and $|\pmb{\omega}|=\omega$. Moreover, in the second equation of \eqref{eq:RE}, the parameter $\mu_j \in \mathbb{R}$ is a constant that adjusts the magnitude and orientation of the vectors on the left- and right-hand sides.
\end{theorem}
\begin{proof}
(i) The relations for the linear momenta follow by inserting \eqref{eq:Ansatz} into the differential equation \eqref{eq:EquationOfMotionReduced} for $\dot{\mathbf q}_{ij}$ and taking into account the derivative rule \eqref{eq:Derivative}. This leads to the following system of linear equations
\begin{equation}
\label{eq:MomentumRelations}
\frac{1}{m_i}\sum_{k\neq i}\mathbf p_{ik}+\frac{1}{m_j}\sum_{k\neq j}\mathbf p_{kj}
=\pmb \omega \times \mathbf q_{ij}, 
\quad i<j\in\{1,\ldots,N\}.
\end{equation}
Now, using the relations \eqref{eq:CasimirsTranslational0}, we reduce to the set of fundamental variables $(\bar{\mathbf{Q}},\bar{\mathbf{P}})$. Notice that any linear relation among the reduced variables extends its validity to the extended set of variables $({\mathbf{Q}},{\mathbf{P}})$ through the relations \eqref{eq:CasimirsTranslational0}. The transformation of \eqref{eq:MomentumRelations} into the fundamental set of variables yields the linear system
\begin{equation}
    \label{eq:LinSyst}
    \mathcal{A}\cdot \bar{\mathbf{P}}=M_{\widehat{\omega}}\, \cdot\bar{\mathbf{Q}},
\end{equation}
where $M_{\widehat{\omega}}$ is a diagonal matrix consisting of $N-1$ repeated blocks of $\widehat{\pmb{\omega}}$, and $\mathcal{A}$ is the $(N-1)\times (N-1)$ diagonal-blocks matrix
\begin{equation}
  \mathcal{A}=  \left(
\begin{array}{cccc}
 \frac{N-1}{m_2}+\frac{1}{m_1} & \frac{1}{m_1}-\frac{1}{m_2} & \dots & \frac{1}{m_1}-\frac{1}{m_2} \\
 \frac{1}{m_1}-\frac{1}{m_3} & \frac{N-1}{m_3}+\frac{1}{m_1} & \dots & \frac{1}{m_1}-\frac{1}{m_3} \\
 \vdots & \vdots & \ddots & \vdots \\
 \frac{1}{m_1}-\frac{1}{m_N} & \dots & \frac{1}{m_1}-\frac{1}{m_N} & \frac{N-1}{m_N}+\frac{1}{m_1}
\end{array}
\right).
\end{equation}
We solve this system in blocks. First, we note that $\det(\mathcal{A})=N^{N-2}m/m_p$, where $m_p=\prod_{i=1}^N m_i$. Hence, system \eqref{eq:LinSyst} can be solved for $\bar{\mathbf{P}}$ by inverting $\mathcal{A}$.

To compute the inverse of $\mathcal{A}$, we exploit its rank-one perturbation structure. Observe that $\mathcal{A}$ can be written as
\[
\mathcal{A} = M\bigl(d I + u v^T\bigr),
\]
where $I$ denotes the $N$-identity matrix, $u=(m_1-m_2,\ldots,m_1-m_N)^T$, $v=(1,\ldots,1)^T\in\mathbb{R}^{N-1}$, $M=\mathrm{Diag}\bigl(1/(m_1-m_2),\ldots,1/(m_1-m_N)\bigr)$, and $d=N\,m_1$. Since $dI$ is invertible, the matrix $dI+uv^T$ is a rank-one perturbation of a diagonal matrix, and its inverse can be computed using the Sherman--Morrison formula \cite{Hager1989}. This yields
\[
(dI+uv^T)^{-1}
=\frac{1}{d}I-\frac{1}{d}\frac{u v^T}{d+v^Tu}.
\]
Consequently, the inverse of $\mathcal{A}$ is given by
\[
\mathcal{A}^{-1}
=(dI+uv^T)^{-1}M^{-1},
\]
which provides an explicit expression for $\bar{\mathbf{P}}$ in terms of $\bar{\mathbf{Q}}$. Precisely, we get
$${\mathbf p}_{1j}=\pmb\omega\times\left[\dfrac{m_1}{N\,m_s}\sum_{k\neq i}{m_k\,\mathbf q}_{1k}+\dfrac{m_j}{N\,m_s}\sum_{k\neq j}{m_k\,\mathbf q}_{kj}\right].$$
Then, this formula is extended to arbitrary $i,j$ using \eqref{eq:recover}.

The expression for the angular momenta is obtained by inserting \eqref{eq:Ansatz} into the differential equation
\eqref{eq:EquationOfMotionReduced} for $\dot {\mathbf a}_i$, which gives $$\widehat{\pmb\omega}\,e^{\widehat{\pmb\omega}t}{\mathbf a}_i=\dfrac{1}{\Theta_{i1}}\,e^{\widehat{\pmb\omega}t}(\pmb{l}_i\times{\mathbf a}_i).$$
Noting that $e^{\widehat{\pmb\omega}t}$ cancels out we obtain the relation $\left(\Theta_{i1}\,{\pmb\omega}-\pmb{l}_i \right)\times {\mathbf a}_i=0$, which implies that ${\mathbf a}_i$ and $\left(\Theta_{i1}\,{\pmb\omega}-\pmb{l}_i \right)$ are either proportional or $\Theta_{i1}\,{\pmb\omega}=\pmb{l}_i $. In the first case, since $\mathbf{a}_i$ is non-zero there is a real number $\lambda_i\neq0$ such that 
$$\pmb{l}_i=\Theta_{i1}\,{\pmb\omega}+\lambda_i{\mathbf a}_i.$$ 
For the case in which $\Theta_{i1}\,{\pmb\omega}=\pmb{l}_i $, we have that $\lambda_i=0$. Moreover, for the general case $\Theta_{i1}\,{\pmb\omega}\neq\pmb{l}_i $, the value of the constant $\lambda_i$ is related to the Casimir $L_{i3}$, the third component of the angular momentum of the $i$th body in the body frame. On the one hand, taking the scalar product of the above equation with ${\mathbf a}_i$ gives
$$L_{i3}=\pmb{l}_i\cdot{\mathbf a}_i= 
\Theta_{i1}\,{\pmb\omega}\cdot{\mathbf a}_i+\lambda_i=
\Theta_{i1} \omega w_i + \lambda_i
$$
On the other hand
$$L_{i3}=\pmb{l}_i\cdot \mathbf{ a}_i=A_i\mathbb I_iA_i^T\,\pmb\omega_i\cdot \mathbf a_i=\left(\Theta_{i1}Id_3+(\Theta_{i3}-\Theta_{i1})\,\mathbf a_i\cdot \mathbf a_i^T\right)\,\pmb\omega_i\cdot \mathbf a_i=\Theta_{i3}\omega w_i$$
and hence $\lambda_i=k_i\,w_i=(\Theta_{i3}-\Theta_{i1})\,\omega\,w_i$ the formula for $k_i=(\Theta_{i3}-\Theta_{i1})\,\omega$ follows.

To prove part (ii) we use the ansatz \eqref{eq:Ansatz} in the equations of motion \eqref{eq:EquationOfMotionReduced} for $\mathbf p_i$ and $\mathbf l_i$, and then insert the results \eqref{eq:REmomenta} just proved to eliminate the momenta. This immediately gives the equation for $\mathbf q_i$. The equation for $\mathbf a_j$ becomes $(k_jw_j \pmb\omega - \partial V/ \partial \mathbf a_j) \times \mathbf a_j = 0$. This can be reformulated by saying that the first factor is parallel to $\mathbf a_j$, and hence the result follows with the constant of proportionality $\mu_j$.
\end{proof}

\begin{remark}
    The case $n=2$, $m=1$ was studied in \cite{Scheeres2006}, where the author gave conditions for relative equilibrium. Notice that the conditions (26) and (27) appearing in \cite{Scheeres2006} are equivalent to Proposition~\ref{propo:RE} and Theorem~\ref{theorem:Cond} $(ii)$ respectively.
\end{remark}

\subsection{Relative equilibria equations}
In what remains, we denote $n=N-1$ and we are going to restrict to the fundamental set of variables $(\bar{\mathbf{Q}},\bar{\mathbf{P}})$. In this regard, we rewrite some of the conditions for relative equilibria.
\begin{Lemma}
    \label{lemma:FundamentalREmomenta}
    The relative equilibria condition given in \eqref{eq:RE} can be expressed in terms of fundamental variables $(\bar{\mathbf{Q}},\bar{\mathbf{P}})$ as follows
    \begin{equation}
        \label{eq:REfundamental}
        \dfrac{1}{N}\,\widehat{\pmb\omega}^2\left[m_i\mathbf{q}_{1i}+{\tau_i}\sum_{k=2}^N{m_k\,\mathbf q}_{1k}\right]=-\dfrac{\partial V}{\partial \mathbf q_{1i}}.
    \end{equation}
where $\tau_i=({m_1-m_i})/{m_s}.$
\end{Lemma}
\begin{proof}
    This expression is obtained by plugging \eqref{eq:recover} into \eqref{eq:RE}.
\end{proof}

Notice that the relations \eqref{eq:REmomenta} tell us that the momenta are obtained from the coordinates. Thus, we are left with \eqref{eq:RE}, which is a system of $3(n+m)$ nonlinear equations for the initial configuration $(\mathbf q_{1i},\mathbf a_j)$, with $i=1,\ldots,n$, and $j=1,\ldots,m$. Moreover, considering the $O(3)$-symmetry of the potential $V$ and hence the Hamiltonian function, equations \eqref{eq:RE} are susceptible of a convenient rearrangement. Precisely, $V$ may be expressed in terms of the invariants associated with the action of $O(3)$ 
in the following manner 
$$V(\mathbf q,\mathbf a)=\tilde V(\pmb \alpha,\pmb \beta,\pmb \gamma),\qquad \pmb \alpha\in\mathbb R^{n^2},\;\pmb \beta\in\mathbb R^{nm},\;\pmb \gamma\in\mathbb R^{m^2-m},$$
where 
$$\pmb \alpha=\{ q_{il}\}_{i,l=1}^n,\quad \pmb \beta=\{ \nu_{ij}\}_{i=1,\ldots, n}^{j=1,\ldots, m}, \quad \pmb \gamma=\{ a_{jk}\}_{j,k=1,j\neq k}^m,$$ 
are vectors with components displayed in lexicographical order according to their sub-indices being the mentioned components given as 
$$ q_{il}=\mathbf q_{1i}\cdot \mathbf q_{1l},\quad  \nu_{ij}=\mathbf q_{1i}\cdot \mathbf a_j,\quad a_{jk}=\mathbf a_j\cdot \mathbf a_k.$$ 

    

This observation shows that ${\partial V}/{\partial \mathbf q_{1j}}$ and ${\partial V}/{\partial \mathbf a_i}$ are linear combinations of position vectors $\mathbf q_{1i}$, orientation vectors $\mathbf a_j$
\begin{gather}
\begin{aligned}
\label{eq:PartialsV}
\dfrac{\partial V}{\partial \mathbf q_{1i}}&=\sum_{l=2}^N a_{il}\mathbf q_{1l}+\sum_{j=1}^m b_{ij}\mathbf a_j\\
	\dfrac{\partial V}{\partial \mathbf a_j}&=\sum_{k=1,k\neq j}^{m} d_{jk}\mathbf a_k+\sum_{i=2}^{N} b_{ij}\mathbf q_{1i},
\end{aligned}
\end{gather}
where $a_{ii}=2\frac{\partial \tilde V}{\partial  q_{ii}}$, $a_{il}=\frac{\partial \tilde V}{\partial  q_{il}}$ if $i\neq l$, $b_{ij}=\frac{\partial \tilde V}{\partial {\nu}_{ij}}$, and $d_{jk}=\frac{\partial \tilde V}{\partial  a_{jk}}$ with $d_{jj}=0$.

Therefore, inserting the partial derivatives of $V$ given in \eqref{eq:PartialsV} in the relations \eqref{eq:RE} we obtain the corresponding equations for $\mathbf q$ and $\mathbf a$ as follows 
\begin{gather}
\begin{aligned}
\label{eq:RE2}
	\dfrac{1}{N}\,\widehat{\pmb\omega}^2\left[m_i\mathbf{q}_{1i}+{\tau_i}\sum_{k=2}^N{m_k\,\mathbf q}_{1k}\right]+\sum_{l=2}^N a_{il}\mathbf q_{1l}+\sum_{j=1}^m b_{ij}\mathbf a_j&=0,\\
	\mu_j\mathbf a_j-k_jw_j\, \pmb\omega+\sum_{k=1,k\neq j}^{m} d_{jk}\mathbf a_k+\sum_{i=2}^{N} b_{ij}\mathbf q_{1i}&=0.
\end{aligned}
\end{gather}

\begin{remark}
Notice that the above system \eqref{eq:RE2} is nonlinear. However, thinking in the partial derivatives of the potential $\tilde V$ as constants depending on the fixed values of the invariants, we may interpret this system as ``linear'' in $\mathbf q_{1i}$, $\mathbf a_j$. At the same time, the nonlinearity is encapsulated in the coefficients as functions of the invariants. 
\end{remark}

This observation is crucial to express the system in its final form, allowing us to discuss conditions for relative equilibria. Precisely, considering the components of the vectorial variables $\mathbf q_{1i}=(x_i,y_i,z_i)$ and $\mathbf a_{j}=(u_j,v_j,w_j)$, we can rearrange system \eqref{eq:RE2} in the following form
\begin{gather}
\begin{aligned}
\label{eq:RE3aux}
		 -\dfrac{\omega^2}{N}\left(\left[
		\begin{array}{c}
			m_i x_i  \\
			m_i y_i  \\
			0  \\
		\end{array}
		\right]+\tau_i\sum_{k=2}^N\left[
		\begin{array}{c}
			m_k x_k  \\
			m_k y_k  \\
			0  \\
		\end{array}
		\right]\right)+\sum_{l=2}^N a_{il}\mathbf q_{1l}+\sum_{j=1}^m b_{ij}\mathbf a_j&=0, \\
		 \sum_{k=1,k\neq j}^{m} d_{jk}\mathbf a_k+\sum_{i=2}^{N} b_{ij}\mathbf q_{1i}+\left[
		\begin{array}{c}
			\mu_j u_j \\
			\mu_j v_j \\
			(\mu_j-\omega\,k_j)w_j
		\end{array}
		\right]&=0,
\end{aligned}
\end{gather}
where $ i=2,\ldots,N$, $j=1,\ldots,m$. Assuming that $a_{il}$, $b_{jk}$ and $d_{ij}$ are constants coefficients depending on the invariants, the variables are coupled in three different groups
$$\xi=(x_2,\ldots,x_N,u_1,\ldots,u_m)^T,\quad \eta=(y_2,\ldots,y_N,v_1,\ldots,v_m)^T, \quad \zeta=(z_2,\ldots,z_N,w_1,\ldots,w_m)^T.$$
Employing the following notation for the variables $\pmb\Xi^T=(\xi^T,\eta^T,\zeta^T)$, the coefficients $\pmb\Sigma^T=(\pmb \alpha^T,\pmb \beta^T,\pmb \gamma^T, \pmb \mu^T)$, the arbitrary parameters $\pmb\mu^T=(\mu_1,\ldots,\mu_m)$, and fixed physical parameters $\mathbf m=(m_2,\ldots,m_N)$, $\pmb k^T=(k_1,\ldots,k_m)$, and  $\omega$, we may decompose \eqref{eq:RE2} in the following diagonal form
\begin{equation}
\label{eq:RE3}
M_{\pmb\Sigma}\,\pmb\Xi=\mathbf{0}_{n+m},\qquad M_{\pmb\Sigma}=Diag[M_1, M_1, M_2],
\end{equation}
where $M_1$, $M_2$ are square $(n+m)$-matrices constructed from blocks as follows
$$
M=\left[\begin{array}{cc}\{a_{il}\}_{n\times n} & \{b_{ij}\}_{n\times m}  \\\{b_{ij}\}_{n\times m}^T & \{d_{jk}\}_{m\times m}\end{array}\right],\; \left.\begin{array}{lll}M_1 & = & M+Diag[(-\dfrac{\omega^2}{N}\mathbf m,\pmb\mu)]+\mathfrak{M} \\M_2 & = & M+Diag[(\mathbf 0_n,\pmb\mu-\omega\pmb k)]\end{array}\right., 
$$
with all indices running in the ranges $i,l=2,\ldots,N$, and $j,k=1,\ldots,m$. Moreover, $\mathfrak{M}$ is the $(n+m)$-squared matrix given by
$$\mathfrak{M}=\left[\begin{array}{cc}\mathfrak{M}_{11} & \mathbf{0}_{n\times m}  \\\mathbf{0}_{m\times n} & \mathbf{0}_{m\times m}\end{array}\right],\quad\mathfrak{M}_{11}=\left[\begin{array}{ccc}\tau_2 m_2&\ldots&\tau_2 m_N  \\\vdots&\ddots&\vdots\\ \tau_Nm_2 &\ldots& \tau_N m_N\end{array}\right],$$
and $\mathbf{0}_r$, $\mathbf{0}_{sr}$ the $r$ and $s\times r$ null vector and matrix respectively. 

The convenience of expressing the relative equation as in \eqref{eq:RE3} will be fully exploited in the following section. However, without further manipulations, the architecture of the above equations provides the following results

\begin{proposition}
\label{propo:Equilibria}
    Consider a relative equilibrium configuration described by the equations~\eqref{eq:RE3}. 

    \begin{itemize}
        \item[(i)] If $Rank(M_1)=m+n$, then all centers of mass and  are placed along the $z$-axis, and bodies are oriented with their symmetry axis along the $w$-axis, sitting one next to the other in a vertical pile. In this case, the configuration is referred to as a piled-up equilibrium.
        
        \item[(ii)] If $Rank(M_2)=m+n$, then $\zeta=0$ and all the position vectors and symmetry axes lie in the same horizontal coordinate plane.

        \item[(iii)] If $Rank(M_1)=m+n-1$, then all the position vectors and orientations lie in a common plane containing the rotation axis $\pmb\omega$. Consequently, each center of mass describes a circular orbit in a plane perpendicular to $\pmb\omega$. The orbital planes may not be the same. Moreover, all the position and orientation vectors are parallel if we further assume that $L_{j3}=0$.
    \end{itemize}
\end{proposition}

\begin{proof}
    The proofs follow directly from the linear arrangement of equations~\eqref{eq:RE3} and the respective ranks of the matrices $M_1$ and $M_2$. In particular:
    
Item $(i)$ imposes $Rank(M_1)=m+n$ in the homogeneous equations \eqref{eq:RE3}, then $x_i=y_i=u_j=v_j=0$ for all $i=2,\ldots,N$ and $j=1,\ldots,m$. Therefore, all centers of mass are placed along the $z$-axis, and bodies are oriented with their symmetry axis along the $w$-axis, as stated.
        
Considering the hypothesis of item $(ii)$ in equations \eqref{eq:RE3} leads to $z_i=w_j=0$ for all $i=2,\ldots,N$ and $j=1,\ldots,m$, which implies that the position vectors and symmetry axes lie in the same horizontal plane.

Finally, if $Rank(M_1)=m+n-1$, then $\xi = \lambda \eta$ and hence the horizontal projections of all vectors are co-linear. That is to say, $(x_i,y_i)=\lambda(u_j,v_j)$ for $i=2,\dots,N$ and $  j=1,\dots,m$. Therefore, all position vectors and orientations lie in a common plane containing $\pmb\omega$. Moreover, considering the relation $L_{j3}=\Theta_{i3}\omega w_i$ derived in the proof of Theorem~\ref{theorem:Cond}, we have that $L_{j3}=0$ implies $w_i=0$ and hence the orientation vectors are parallel.
\end{proof}

\begin{remark}
The configurations described in item (i) are called piled-up equilibria. For $N$ small and spherical bodies, piled-up equilibria have been reported in \cite{Scheeres2012}.

In a relative equilibrium, the orbital motion of the bodies may lie in different parallel planes, leading to two types of equilibria: Lagrangian (same plane) and non-Lagrangian (different parallel planes). These types of equilibria have been discovered and rediscovered in the literature \cite{Abolenaga1979,Barkin1985,WangA,WangB,Maciejewski1995,Maciejewski2010} under various names, such as great-circle, cylindrical precession, or coplanar precession for Lagrangian equilibria, and non-great-circle or conic precession for non-Lagrangian ones. The configurations described in item (ii) correspond to the Lagrangian type.

Item (iii) generalizes the co-linear solution of the $N$-body problem for point masses.
\end{remark}

The rank of $M_1$ can be adjusted with the parameter $\omega$ or $\pmb\mu$. Fixing the potential and the coefficients of its derivative requires fixing the $SO(3)$ invariants of the position and orientation vectors, that is the set of all their scalar products. After that there is still the freedom to rotate all vectors. Since we already fixed the rotation axis $\pmb\omega \parallel \mathbf e_z$ the remaining freedom is to rotate about the $z$-axis. 

\section{Necessary conditions for the existence of relative equilibria}
In this section, we analyze the solvability of system \eqref{eq:RE3}, which will provide us with necessary conditions for the existence of solutions. For this purpose, we will interchange the roles of $\pmb\Xi$ and $\pmb\Sigma$. Hence, we consider the equivalent system 
\begin{equation}
\label{eq:REinverted}
M_{\pmb\Xi}\,\pmb\Sigma=\pmb\varrho_{\pmb\Xi}
\end{equation}
where
$$M_{\pmb\Xi}=\dfrac{\partial (M_{\pmb\Sigma}\pmb\Xi)}{\partial\, \pmb\Sigma},\quad \pmb\varrho_{\pmb\Xi}=M_{\pmb\Sigma}\,\pmb\Xi-M_{\pmb\Xi}\,\pmb\Sigma.$$
In system \eqref{eq:REinverted}, we consider the coefficients $\pmb\Sigma$ as the indeterminate variables and the old variables $\pmb\Xi$ as coefficients in system \eqref{eq:REinverted}. This manipulations allow us to provide necessary conditions regarding the compatibility of system \eqref{eq:RE3}. The matrix $M_\Xi$ only depends on $\pmb\Xi$, all other parameters $\omega \pmb k$ and $\omega^2 \mathbf m$ appear on the right hand side $\pmb\rho_\Xi$, along with $x_i, y_i$, and $w_j$. The purpose of this rewriting is to obtain universal solvability conditions, i.e., conditions on the variables $\pmb\Xi$ that are necessary for any potential with the assumed $SO(3)$ symmetry.

\begin{Proposition}
Let $\pmb\Xi$ be a relative equilibria configuration, and let $W_{\pmb\Xi}$ be the vectorial space generated by the columns of the matrix $M_{\pmb\Xi}$. Then, for every vector $\mathbf v_{\pmb\Xi}\in W_{\pmb\Xi}^\bot$ we have 
$$\langle \pmb\varrho_{\pmb\Xi},\mathbf v_{\pmb\Xi}\rangle=0.$$
\end{Proposition}
\begin{Proof}
It is a standard fact from linear algebra that the solvability condition for a linear inhomogeneous equation is that the right hand side $\pmb\rho_\Xi$ must be orthogonal to $W_\Xi^\perp = Kern( M_\Xi^T )$.
\end{Proof}

In practice, we will employ the above proposition to obtain the mentioned necessary conditions. Precisely, we proceed in the following way: for a given basis $\{\mathbf v_{\pmb\Xi}^1,\ldots,\mathbf v_{\pmb\Xi}^r \}$ of  $W_{\pmb\Xi}^\bot$, we have the following $r$ necessary conditions for the vector $\pmb\Xi$ being a relative equilibria 
\begin{equation}
\label{eq:NecCond}
\langle \pmb\varrho_{\pmb\Xi},\mathbf v_{\pmb\Xi}^i\rangle=0,\quad i=1,\ldots,r.
\end{equation}
The relative equilibrium conditions \eqref{eq:RE} are a system of $3(n+m) $ non-linear equations in $3(n+m)$ variables. We interpret it as a homogeneous system of linear equations \eqref{eq:RE3} in $3(n+m)$ variables $\pmb\Xi$, and as an inhomogeneous system of $3(n+m)$ linear equations \eqref{eq:REinverted} in $n^2+nm+m^2$ variables $\pmb\Sigma$. For $N=n+1=2$, $m=1,2$  and for $N=n+1=3$, $m=1$ this is useful because the kernel of $M_\Xi$ can reveal universal conditions on the solutions. For $N=3$ and $m=0$, i.e. the usual 3-body problem, the result of this approach is the well known statement that rotating relative equilibria are planar.
\subsection{General procedure for necessary conditions}
 \label{sec:GeneralProc}
Employing the previous development, we propose the following procedure to obtain necessary conditions for relative equilibria in a system consisting of $n$ bodies. We will assume that there are $m$ bodies ($m\leq n$) with an axial-symmetric mass distribution and $n-m$ with a spherically symmetric mass distribution. Considering Theorem~\ref{theorem:Cond}, we only have to impose conditions on the configuration space of the translational and rotational variables $\mathbf q_i$ and  $\mathbf a_j$, where $i=1,\ldots,n$ and $j=1,\ldots,m$. The following notation will be useful in separating our analysis in several cases
$$\mathcal A:=\{\mathbf q_{12},\ldots,\mathbf q_{1N},\mathbf a_1,\ldots,\mathbf a_m,\mathbf e_3\},$$
where $\mathbf e_3=(0,0,1)^T$, and we refer to the $i$-th element in $\mathcal A$ as $a_i$, with $i=1,\ldots,n+m+1$. Then, given a set of natural ordered numbers $\mathcal I\subset\{1,\ldots,n+m+1\}$ satisfying that $|\mathcal I|\geq2$, we define $\mathcal A_{\mathcal I}\subset \mathcal A$ as $\mathcal A_{\mathcal I}:=\bigcup_{i\in\mathcal I} a_i,$ and we consider the dimension of each subspace $span\{\mathcal A_{\mathcal I}\}$
$$\rho(\mathcal I):=Dim(span\{\mathcal A_{\mathcal I}\}).$$
Notice that there are exactly $h:=2^{n+m+1}-(n+m+2)$ different possibilities for $\mathcal I$. Now, we give the 4-step procedure to identify allowed configurations.

\begin{description}
  \item[Step 1] We distinguish each possible case by fixing all the dimensions $(\rho(\mathcal I_1),\ldots,\rho(\mathcal I_h))$.
  \item[Step 2] For each case, we impose the fixed dimension conditions $(\rho(\mathcal I_1),\ldots,\rho(\mathcal I_h))$ on system \eqref{eq:REinverted} and compute a basis $\{\mathbf v_{\pmb\Xi}^1,\ldots,\mathbf v_{\pmb\Xi}^r \}$ of the null-space  associated with the columns of matrix $M_{\pmb\Xi}$.
  \item[Step 3] By computing \eqref{eq:NecCond} we obtain necessary conditions on the variables.
  \item[Step 4] After assuming conditions of step 3 in system \eqref{eq:REinverted}, we check if there are further restrictions to obtain well defined solutions for \eqref{eq:REinverted}.
\end{description}

\begin{remark}
 Notice that steps 3 and 4 simplify considerably for the case $\mathbf s=0$. Precisely, all the involved operations become linear for these steps. 
\end{remark}

\section{Analysis of a sphere and an axial-symmetric body}
 \label{sec:SphereAxSyBodyAnalysis}
 We analyze the case of two bodies, one with spherical symmetry and the other with mass axially symmetrically distributed. Under these assumptions $n=m=1$, $h=2^{n+m+1}-(n+m+2)=4$, and the position and orientation vectors may be expressed without sub-indices as $\mathbf q=(x,y,z)$ and $\mathbf a=(u,v,w)$. Thus, we have that $\mathcal A=\mathcal A_{\mathcal I_1}=\{\mathbf q,\mathbf a,\mathbf e_3\}$, $\mathcal A_{\mathcal I_2}=\{\mathbf q,\mathbf a\}$, $\mathcal A_{\mathcal I_3}=\{\mathbf q,\mathbf e_3\}$ and $\mathcal A_{\mathcal I_4}=\{\mathbf a,\mathbf e_3\}$.
 
In \cite{Kinoshita1970}, the author considered a second-order Legendre expansion of the potential of a sphere and an axial-symmetric ellipsoid and proved the existence of several types of relative equilibria. In particular, Kinoshita classified the relative equilibria in three categories: spoke, float, and arrow relative equilibria, which correspond with cases 3, 4, and 6, respectively, in our following development. The list given by \cite{Kinoshita1970} needs to be completed. \cite{Scheeres2012} studied equilibria in which the bodies touch each other; they correspond with our cases 1 and 2. All these relative equilibria are Lagrangian, i.e., the centers of mass of all bodies lie in the same plane. Nevertheless, relative equilibria exist where the plane containing the curve described by the center of mass of one body does not contain the center of mass of the other; these are non-Lagrangian relative equilibria. The existence of the non-Lagrangian type took a long debate until these equilibria were discovered and rediscovered in several works \cite{Abolenaga1979,Barkin1985,WangA,WangB,Maciejewski2010}. 
 
Now, we carry out Step 1 from the procedure described above and study each of the possible combinations for the dimensions  $\rho(\mathcal I_i)$ given in Table~\ref{table:T1}. Then, we will apply each of the above cases to system \eqref{eq:REinverted}, which for the case of a sphere and an axial-symmetric body reads as follows 
\begin{equation}
    \label{eq:REinvertedSphereAxisB}
    \left(\begin{array}{ccc}
 2 x & u & 0 \\
 0 & x & u \\
 2 y & v & 0 \\
 0 & y & v \\
 2 z & w & 0 \\
 0 & z & w \\
\end{array}
\right)\, \left(\begin{array}{c}
 a_{22}\\
 b_{2 1}\\
 \mu_1 \\
\end{array}
\right)= \left(\begin{array}{c} -\omega^2m_{12} x\\
0\\
-\omega^2 m_{12}  y\\
0\\
0\\
-\omega\,k_1\, w\\
\end{array}
\right),
\end{equation}
where $m_{12}=m_1m_2/(m_1+m_2)$.

In the subsequent discussion of the cases listed in Table~\ref{table:T1}, we will consider the following notation: by the symbol $\mathbf e_i^j$ we refer to the $i$-th vector of the canonical base of $\mathbb R^j$. 
    \begin{table}[h]
	\begin{center}
		\begin{tabular}{| c| c | c |c|c|}\hline
	$N=n+1=2,\,\,m=1$		  & $\rho(\{\mathbf q,\mathbf a,\mathbf e_3\})$ & $\rho(\{\mathbf q,\mathbf a\})$&$\rho(\{\mathbf q,\mathbf e_3\})$&$\rho(\{\mathbf a,\mathbf e_3\})$ \\ \hline\hline
			Case 1: arrow &3 & 2&2&2 \\\hline
			Case 2: conic &2 & 2&2&2 \\\hline
			Case 3: float &2 & 2&2&1 \\\hline
			Case 4: spoke &2 & 1&2&2 \\\hline
			Case 5: piled up &2 & 2&1&2 \\\hline
			Case 6: piled up &1 & 1&1&1 \\\hline
		\end{tabular}
	\end{center}
    \caption{Combinations for the dimensions  $\rho(\mathcal I_i)$ in the case of a sphere and an axial-symmetric body.}
 \label{table:T1}
\end{table}
  
 \begin{description}
\item[Case 1] In this case all three vectors are independent, which does not imposes extra conditions on system  \eqref{eq:REinverted}. Hence, carrying out step 2 implies computing the null-space associated to matrix $M_\Xi$ in system \eqref{eq:REinvertedSphereAxisB}. This vectorial space is generated by the following vectors
 $$\mathbf v_{\pmb\Xi}^1= \frac{y\, h_2}{u\,h_3}\mathbf e_1^6 -\frac{w}{u}\mathbf e_2^6-\frac{x\, h_2}{u\,h_3} \mathbf e_3^6+\mathbf e_6^6,\quad \mathbf v_{\pmb\Xi}^2= \frac{ h_1}{u\,h_3}\mathbf e_1^6 -\frac{h_1}{u\,h_3} \mathbf e_3^6+\mathbf e_5^6,\quad \mathbf v_{\pmb\Xi}^3=-\frac{y}{u}\mathbf e_1^6-\frac{v}{u}\mathbf e_2^6+\frac{x}{u}\mathbf e_3^6+\mathbf e_4^6,$$
 where $(h_1,h_2,h_3)=\mathbf q\times \mathbf a$. Moreover, the right hand side in system  \eqref{eq:REinverted} is  $\pmb\varrho_{\pmb\Xi}=\omega^2( m_{12}( x\,\mathbf e_1^6+ y\,\mathbf e_3^6)-k_1/\omega\, w\, \mathbf e_6^6).$ Then, we compute the necessary conditions from step~3 
 $$\langle \pmb\varrho_{\pmb\Xi},\mathbf v_{\pmb\Xi}^1\rangle=-\omega \,k\,w=0,\quad \langle\pmb\varrho_{\pmb\Xi},\mathbf v_{\pmb\Xi}^2\rangle=-\omega ^2 m_{12} z=0,\quad \langle\pmb\varrho_{\pmb\Xi},\mathbf v_{\pmb\Xi}^3\rangle=0,$$
 which determine the conditions $z=w=0$. 

 Finally, we execute step 4. It requires imposing the conditions $z=w=0$ in system~\eqref{eq:RE3}, which leads to the following equations:
\begin{gather}
    \begin{aligned}
        x  \left(2 a_{22}- m_{12}   \omega ^2\right)+u\,  b_{21}=&\,0,\quad
        x\,  b_{21}+\mu_1\,   u =\,0,\\
        y  \left(2 a_{22}- m_{12}   \omega ^2\right)+v \, b_{21}=&\,0,\quad
        y\,  b_{21}+\mu_1\,   v =0.
    \end{aligned}
\end{gather}
 Solving the above equation in the coefficients $(a_{22},b_{21},\mu_1)$, we finally obtain the following relations
\begin{equation}
    \label{eq:ExtraCaso1}
    a_{22}= \frac{1}{2} m_{12} \omega ^2,\quad b_{21}=\mu_1= 0.
\end{equation}

 Summarizing, we have that case 1 imposes conditions \eqref{eq:ExtraCaso1}, which can be verified once the potential function is specified, and $z=w=0$. Therefore, this type of relative equilibria are restricted to be of the Lagrangian type. For the special cases in which $\mathbf q $ and $ \mathbf a$ are aligned or perpendicular, Kinoshita \cite{Kinoshita1970} reported these equilibria as spoke and arrow respectively. In \cite{Maciejewski2010} the authors report \emph{inclined coplanar precessions}, which are of the Lagrangian type, with the symmetry axis $\mathbf a$ pointing in an arbitrary direction in space. However, our analysis shows that $w=0$, implying that $\mathbf a$ must lie in the orbital plane rather than having a spatial arbitrary inclination.
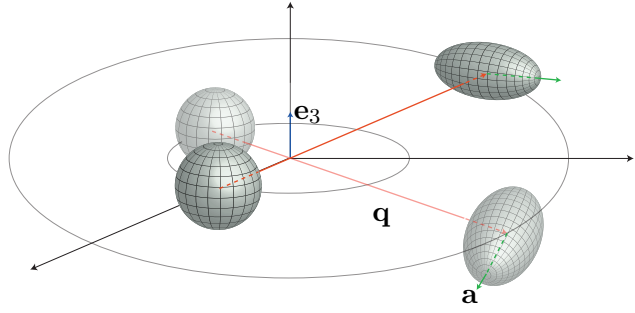
\begin{figure}[]
		\begin{center}
			\begin{tikzpicture}[scale=0.9]
				\node at (0,0) {\includegraphics[width=240pt]{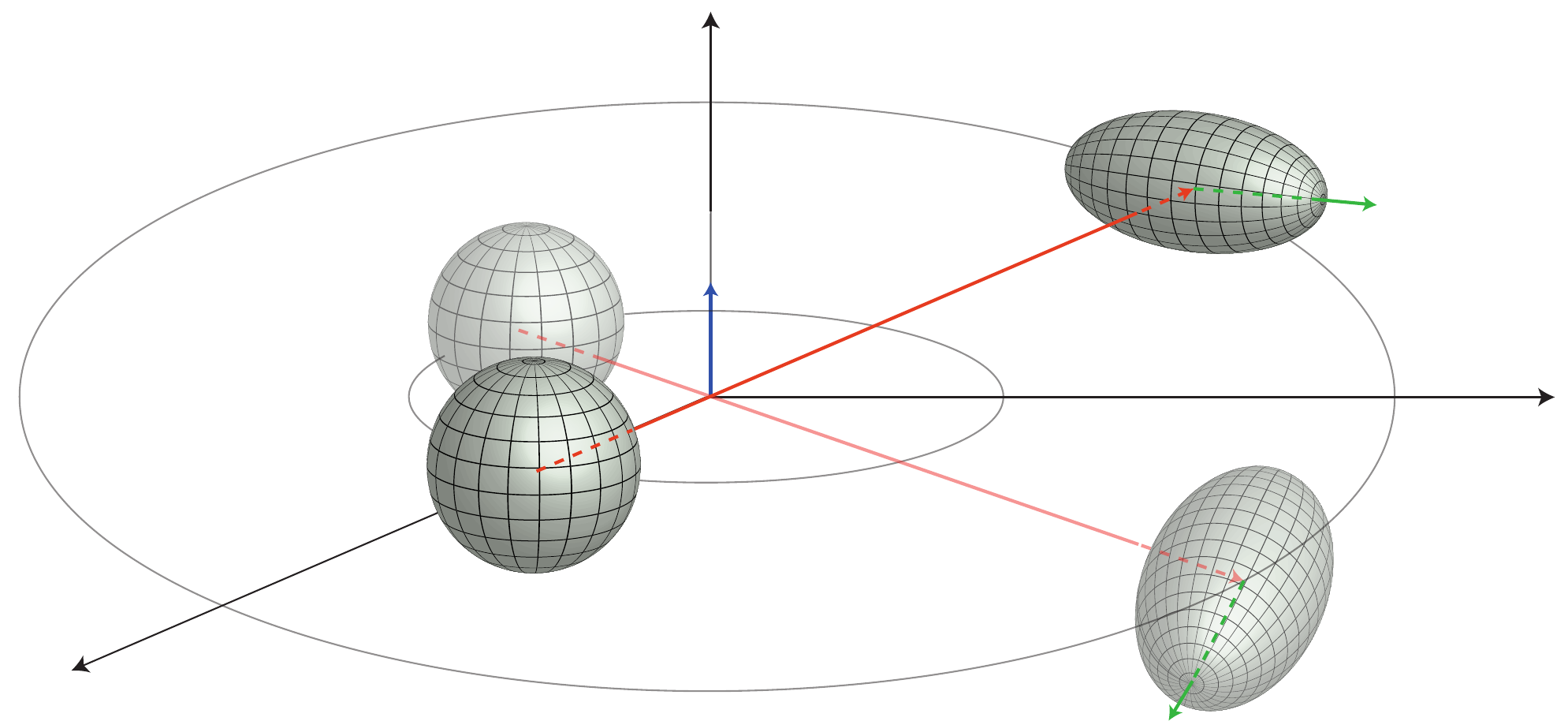}}; 
				\coordinate [label=below:\textcolor{black}{$\mathbf e_3$}] (E2) at (-0.18,0.7);
				\coordinate [label=below:\textcolor{black}{$\mathbf q$}] (E2) at (0.9,-0.8);
				\coordinate [label=below:\textcolor{black}{$\mathbf a$}] (E2) at (2.2,-2.0);
			\end{tikzpicture}
		\end{center}
		\caption{Case 1. Extra conditions $w= z =0$. Only Lagrangian equilibria type. When $\mathbf q\cdot \mathbf a=0$, this configuration corresponds with the arrow equilibria reported in \cite{Kinoshita1970}. }
		\label{fig:Arrow}
	\end{figure}

 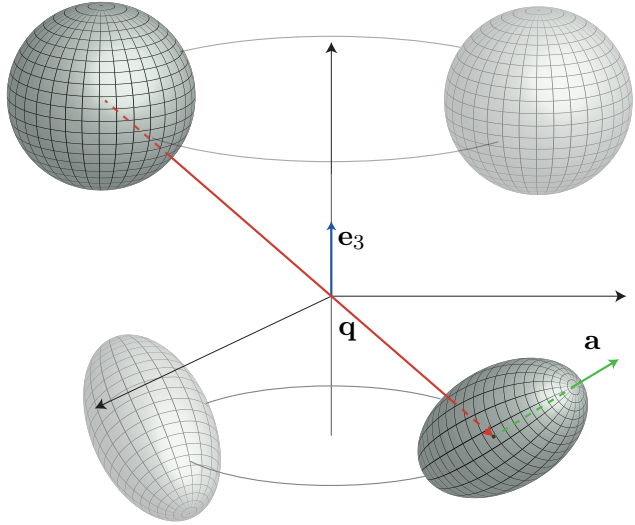
\begin{figure}[h!]
		\begin{center}
			\begin{tikzpicture}[scale=0.6]
				\node at (0,0) {\includegraphics[width=240pt]{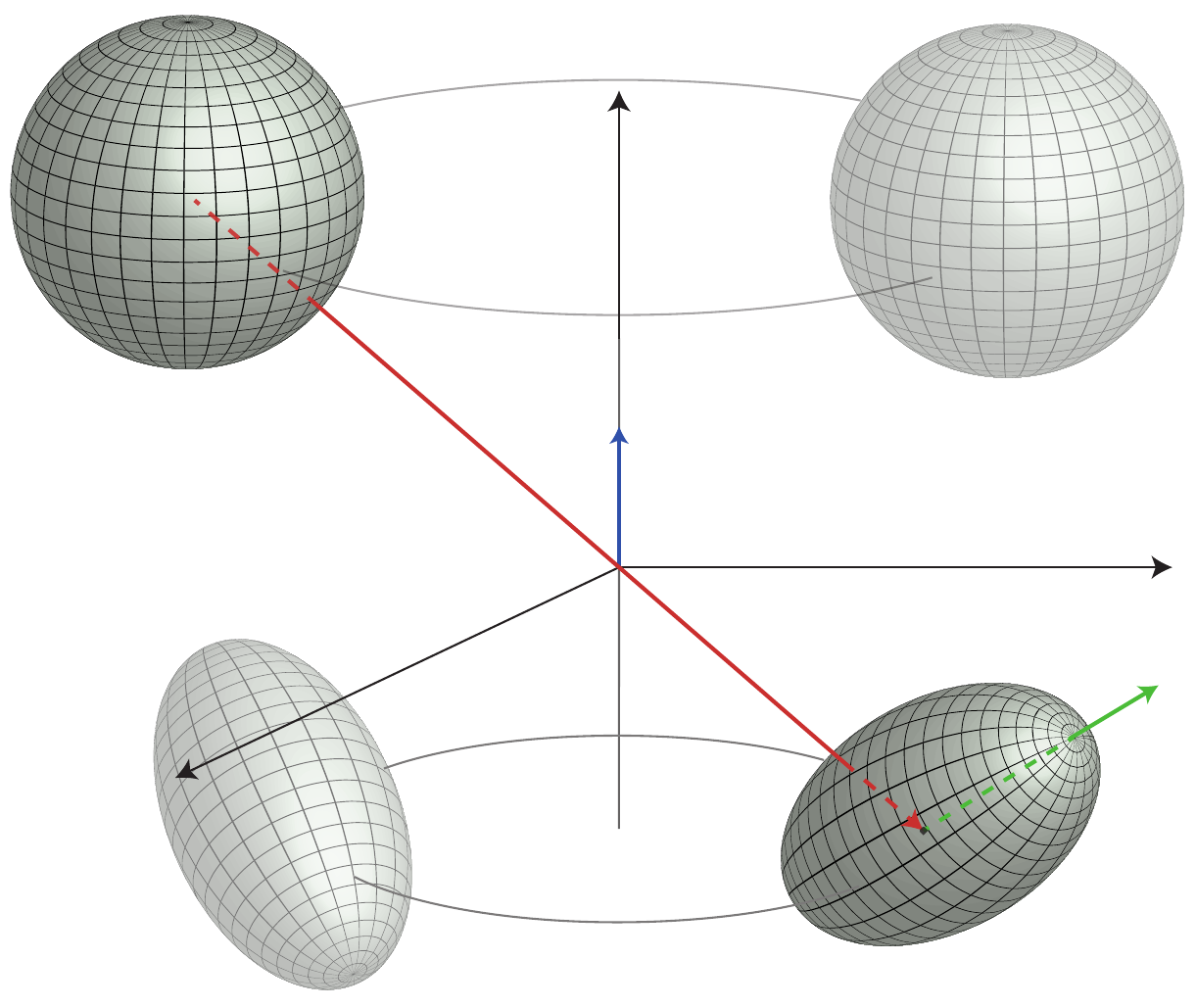}}; 
				\coordinate [label=below:\textcolor{black}{$\mathbf e_3$}] (E2) at (0.7,0.9);
				\coordinate [label=below:\textcolor{black}{$\mathbf q$}] (E2) at (0.6,-1.1);
				\coordinate [label=below:\textcolor{black}{$\mathbf a$}] (E2) at (6.0,-1.4);
			\end{tikzpicture}
		\end{center}
		\caption{Conic equilibria, case 2, $\mathbf q \times \mathbf a \perp \mathbf e_3$. Extra condition $w\,k_1/\omega-\lambda\,m_{12}\, z=0$. Non-Lagrangian type allowed.}
		\label{fig:Conic}
	\end{figure}
 \item[Case 2] This case implies that the three vectors $(\mathbf q, \mathbf a, \mathbf e_3)$ are on the same plane, but also there is no any pair of parallel vectors. These requirements impose that $(x,y)=\lambda(u,v)$ for some $\lambda\in\mathbb R$, $z\neq \lambda w$, and $\mathbf q \times \mathbf a \perp \mathbf e_3$. Considering these conditions on matrix $M_\Xi$ in system \eqref{eq:REinvertedSphereAxisB} leads to 
 $$M_\Xi=\left(
\begin{array}{ccc}
 2 \lambda  u  & u  & 0 \\
 0 & \lambda  u  & u  \\
 2 \lambda  v  & v  & 0 \\
 0 & \lambda  v  & v  \\
 2 z  & w  & 0 \\
 0 & z  & w  \\
\end{array}
\right).$$
The right hand side of system \eqref{eq:REinvertedSphereAxisB} is $\pmb\varrho_{\pmb\Xi}=\omega^2 \lambda\,m_{12}( u\,\mathbf e_1^6+ v\,\mathbf e_3^6-\omega\,k_1\, w\, \mathbf e_6^6)$, and the corresponding 3-dimensional null-space of matrix $M_\Xi$ is 
  $$\mathbf v_{\pmb\Xi}^1= -\frac{z}{u}\mathbf e_1^6 -\frac{w}{u}\mathbf e_2^6+\lambda \mathbf e_5^6+\mathbf e_6^6,\quad \mathbf v_{\pmb\Xi}^2= -\frac{v}{u} \mathbf e_2^6+\mathbf e_4^6,\quad \mathbf v_{\pmb\Xi}^3=-\frac{v}{u}\mathbf e_1^6+\mathbf e_3^6.$$
Thus, the null-space condition becomes
 $$\langle \pmb\varrho_{\pmb\Xi},\mathbf v_{\pmb\Xi}^1\rangle=\omega ^2 (\lambda \,m_{12} z-w\,k_1/\omega)=0,\quad \langle \pmb\varrho_{\pmb\Xi},\mathbf v_{\pmb\Xi}^2\rangle=\langle \pmb\varrho_{\pmb\Xi},\mathbf v_{\pmb\Xi}^3\rangle=0.$$
That is to say, $\lambda \,m_{12} z-w\,k_1/\omega=0$ is a necessary condition for the existence of this kind of relative equilibria.

Again, step 4 requires imposing the condition $\lambda \,m_{12} z-w\,k_1/\omega=0$ in system~\eqref{eq:RE3}, and solving this system for the coefficients $(a_{22},b_{21},\mu_1)$. However, this operation does not provide new conditions.

This case complements the previous literature, since non-Lagrangian equilibria only happen for the particular relation $z=\frac{k_1\, w}{\lambda\,\omega  \,m_{12}  }$, which has not been reported before. We illustrate this configuration in Figure~\ref{fig:Conic}.

 \begin{figure}[h!]
		\begin{center}
			\begin{tikzpicture}[scale=0.85]
				\node at (0,0) {\includegraphics[width=230pt]{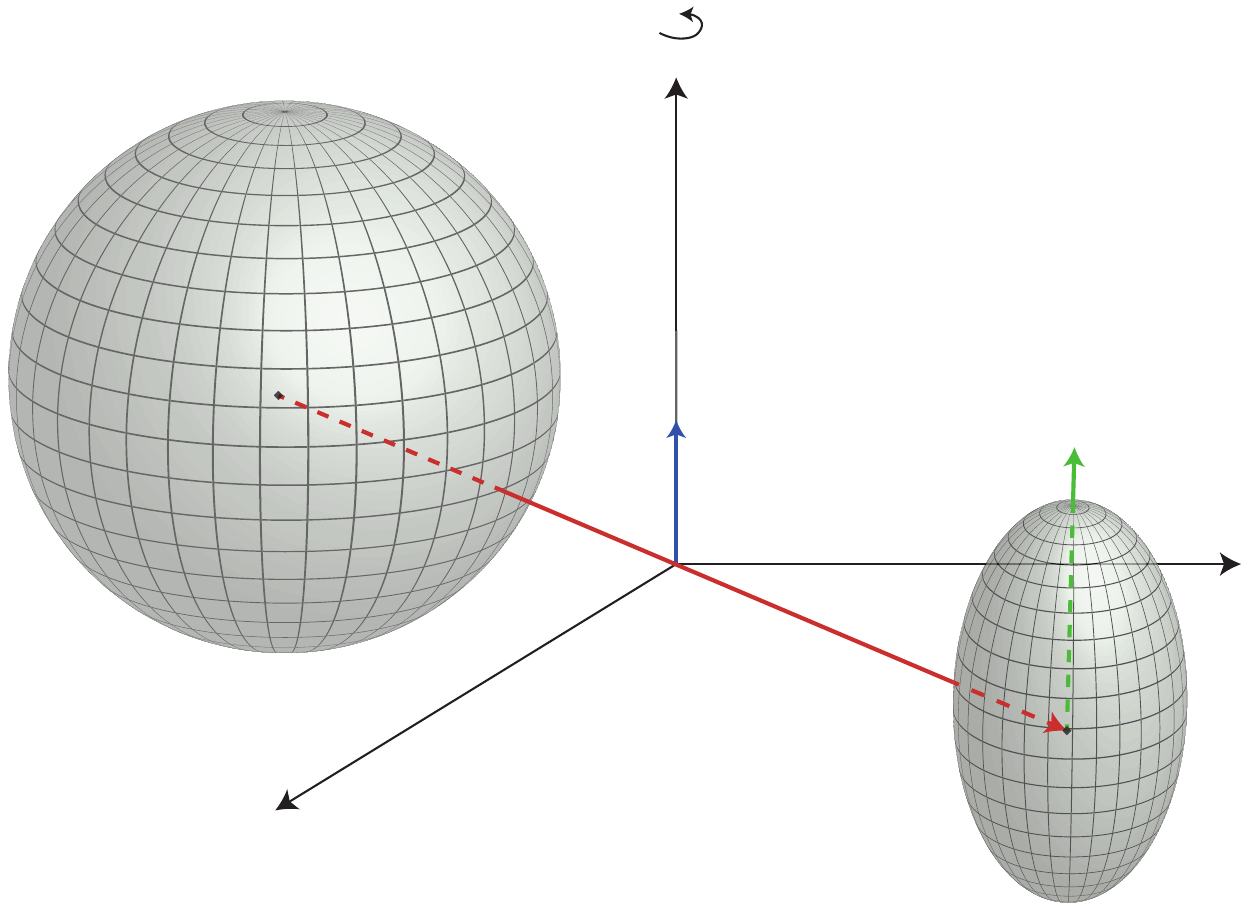}}; 
				\coordinate [label=below:\textcolor{black}{$\mathbf e_3$}] (E2) at (0.9,0.6);
				\coordinate [label=below:\textcolor{black}{$\mathbf q$}] (E2) at (1.8,-1.7);
				\coordinate [label=below:\textcolor{black}{$\mathbf a$}] (E2) at (3.8,0.7);
			\end{tikzpicture}
		\end{center}
		\caption{Float equilibria, case 4, $\mathbf a \parallel \mathbf e_z$. Extra condition $z=0$. Only Lagrangian equilibria are allowed.}
		\label{fig:Float}
	\end{figure}
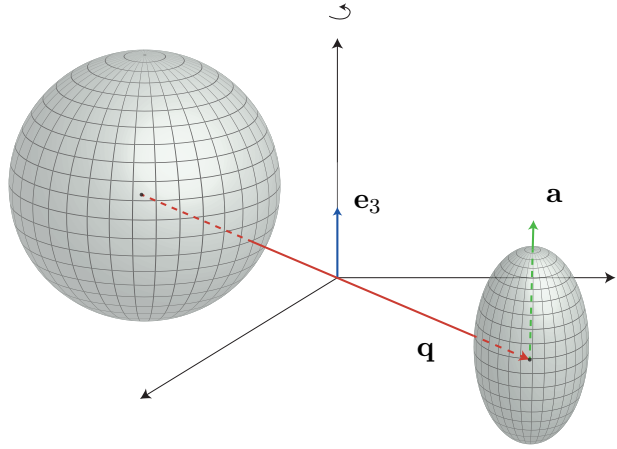 
 \item[Case 3] The dimensions condition requires $\mathbf a \parallel \mathbf e_z$, hence $u=v=0$. Assuming $x\neq0$, we have $\pmb\varrho_{\pmb\Xi}=\omega^2( m_{12}( x\,\mathbf e_1^6+ y\,\mathbf e_3^6)-k_1/\omega\, w\, \mathbf e_6^6)$, and the corresponding null-space is spanned by  
  $$\mathbf v_{\pmb\Xi}^1= -\frac{z}{x}\mathbf e_1^6 -\frac{w}{x}\mathbf e_2^6+\mathbf e_6^6,\quad \mathbf v_{\pmb\Xi}^2= -\frac{y}{x} \mathbf e_2^6+\mathbf e_4^6,\quad \mathbf v_{\pmb\Xi}^3=-\frac{y}{x}\mathbf e_1^6+\mathbf e_3^6.$$
Proceeding with Step 3 leads to the relation $\langle \pmb\varrho_{\pmb\Xi},\mathbf v_{\pmb\Xi}^1\rangle=-\omega ^2 m_{12}\,z=0.$ Moreover, proceeding as in the previous cases, step 4 provides the following conditions
$$a_{22}= \frac{1}{2} m_{12} \omega ^2,\quad b_{21}= 0,\quad \mu_1= k_1\, \omega .$$

It implies that this type of relative equilibrium is of the Lagrangian type, namely, we have the condition $z=0$. In Figure~\ref{fig:Float} we illustrate this configuration, which was named float equilibrium by Kinoshita \cite{Kinoshita1970}. 

 \begin{figure}[h!]
		\begin{center}
			\begin{tikzpicture}[scale=0.8]
				\node at (0,0) {\includegraphics[width=230pt]{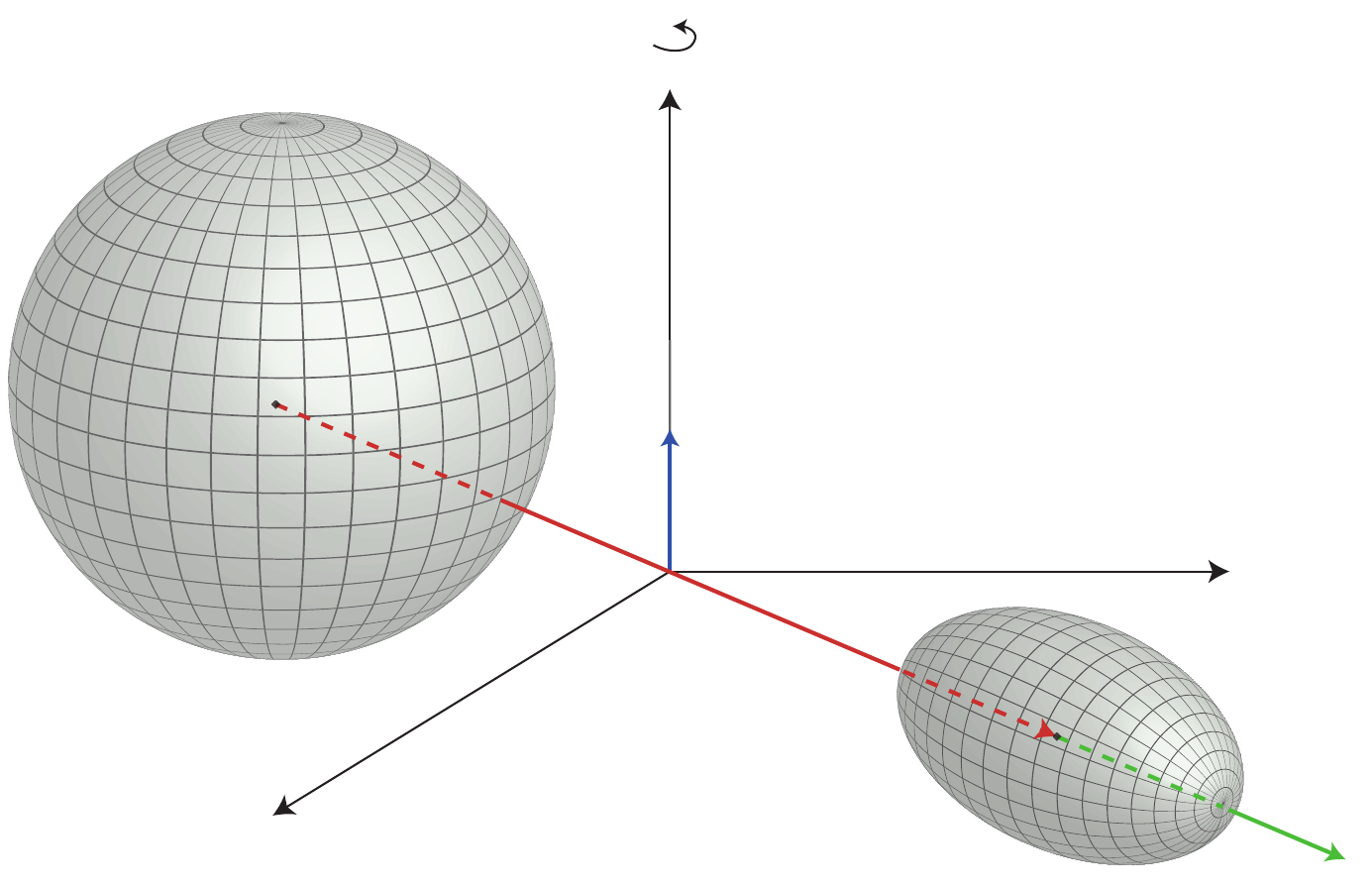}}; 
				\coordinate [label=below:\textcolor{black}{$\mathbf e_3$}] (E2) at (0.2,0.55);
				\coordinate [label=below:\textcolor{black}{$\mathbf q$}] (E2) at (1.1,-0.85);
				\coordinate [label=below:\textcolor{black}{$\mathbf a$}] (E2) at (4.5,-2.0);
			\end{tikzpicture}
		\end{center}
	\caption{Spoke equilibria, case 3, $\mathbf q \parallel \mathbf a$. Extra conditions $w=z=0$. Only Lagrangian equilibria are allowed.}
	\label{fig:Spoke}
	\end{figure}
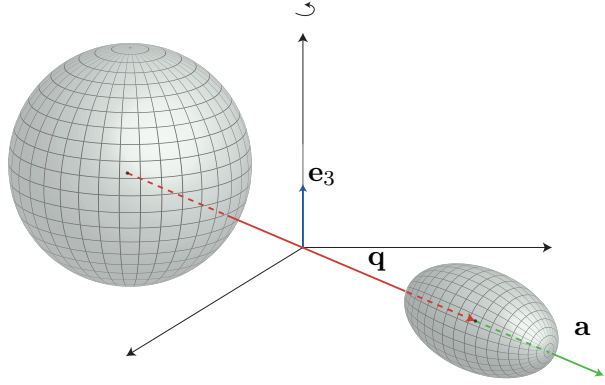

  \item[Case 4] In this case $\mathbf q = \lambda \mathbf a$ for some non-zero $\lambda\in\mathbb R$. As a result the null-space is a 4-dimensional space with basis $\{\mathbf v_{\pmb\Xi}^1,\ldots,\mathbf v_{\pmb\Xi}^4 \}$. Assuming $u\neq0$, these vectors read as follows
$$\mathbf v_{\pmb\Xi}^1=-\frac{w }{u }\mathbf e_2^6+\mathbf e_6^6,\quad \mathbf v_{\pmb\Xi}^2=-\frac{w }{u }\mathbf e_1^6+\mathbf e_5^6,\quad \mathbf v_{\pmb\Xi}^3=-\frac{v }{u }\mathbf e_2^6+\mathbf e_4^6,\quad \mathbf v_{\pmb\Xi}^4=-\frac{v }{u }\mathbf e_1^6+\mathbf e_3^6,$$
and $\pmb\varrho_{\pmb\Xi}=\omega ^2( \lambda\,   m_{12}   u  ,0, \lambda\,   m_{12}   v  ,0,0,-k_1/\omega\,   w)$. Proceeding with Step 3 leads to the following relations
$$\langle \pmb\varrho_{\pmb\Xi},\mathbf v_{\pmb\Xi}^1\rangle=-\omega \,k_1\,w=0, \quad \langle \pmb\varrho_{\pmb\Xi},\mathbf v_{\pmb\Xi}^2\rangle=-\lambda\,\omega ^2m_{12}\,w=0,\quad \langle \pmb\varrho_{\pmb\Xi},\mathbf v_{\pmb\Xi}^3\rangle=\langle \pmb\varrho_{\pmb\Xi},\mathbf v_{\pmb\Xi}^4\rangle=0,$$
which implies $w=0$ and because $\mathbf q=\lambda \mathbf a$ also  $z=0$. Moreover, Step 4 does not add new constraints. Figure~\ref{fig:Spoke} illustrates this configuration, which has been named as spoke equilibrium \cite{Kinoshita1970}.

  \item[Case 5] The dimension condition requires $\mathbf q \parallel \mathbf e_z$, hence $x=y=0$. As a result the null-space is a 3-dimensional space generated by $\{\mathbf v_{\pmb\Xi}^1,\mathbf v_{\pmb\Xi}^2,\mathbf v_{\pmb\Xi}^3 \}$. Assuming $u\neq0$, these vectors and $\pmb\varrho_{\pmb\Xi}$ read as follows
  $$\mathbf v_{\pmb\Xi}^1= -\frac{z}{u}\mathbf e_1^6 -\frac{w}{u}\mathbf e_2^6+\mathbf e_6^6,\quad \mathbf v_{\pmb\Xi}^2= -\frac{v}{u} \mathbf e_2^6+\mathbf e_4^6,\quad \mathbf v_{\pmb\Xi}^3=-\frac{v}{u}\mathbf e_1^6+\mathbf e_3^6,\quad \pmb\varrho_{\pmb\Xi}=-\omega\,k_1\, w\, \mathbf e_6^6.$$
Proceeding with Step 3 leads to the following relations
$$\langle \pmb\varrho_{\pmb\Xi},\mathbf v_{\pmb\Xi}^1\rangle=-\omega\,k_1\,w=0,\quad \langle \pmb\varrho_{\pmb\Xi},\mathbf v_{\pmb\Xi}^2\rangle=\langle \pmb\varrho_{\pmb\Xi},\mathbf v_{\pmb\Xi}^3\rangle=0.$$
Moreover, Step 4 imposes the following conditions on the coefficients
$$a_{22}=b_{21}= \mu_1= 0.$$ 
Therefore, for this type of relative equilibria we must impose $w=0$, and the bodies are piled-up with the symmetry axis $\mathbf a$ perpendicular to $\mathbf q$ and $\mathbf e_3$, as it is illustrated in the right image of Figure~\ref{fig:Piled-up1}.

  \begin{figure}[h]
		\begin{center}
			\begin{tikzpicture}[scale=0.8]
				\node at (0,0) {\includegraphics[width=145pt]{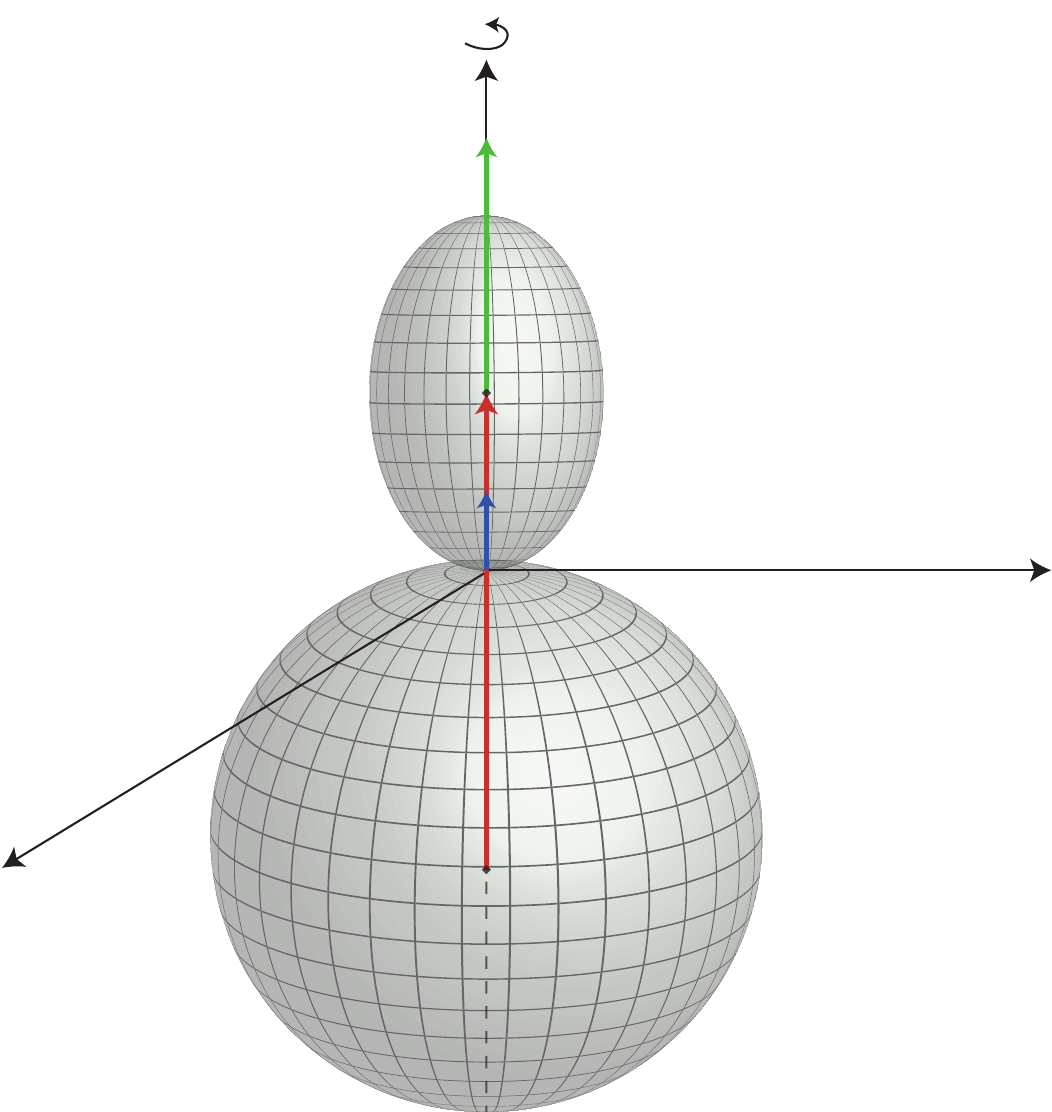}}; 
				\coordinate [label=below:\textcolor{black}{$\mathbf e_3$}] (E2) at (0.1,0.5);
				\coordinate [label=below:\textcolor{black}{$\mathbf q$}] (E2) at (0,1.3);
				\coordinate [label=below:\textcolor{black}{$\mathbf a$}] (E2) at (0,2);
			\end{tikzpicture}\quad 
			\begin{tikzpicture}[scale=0.8]
				\node at (0,0) {\includegraphics[width=145pt]{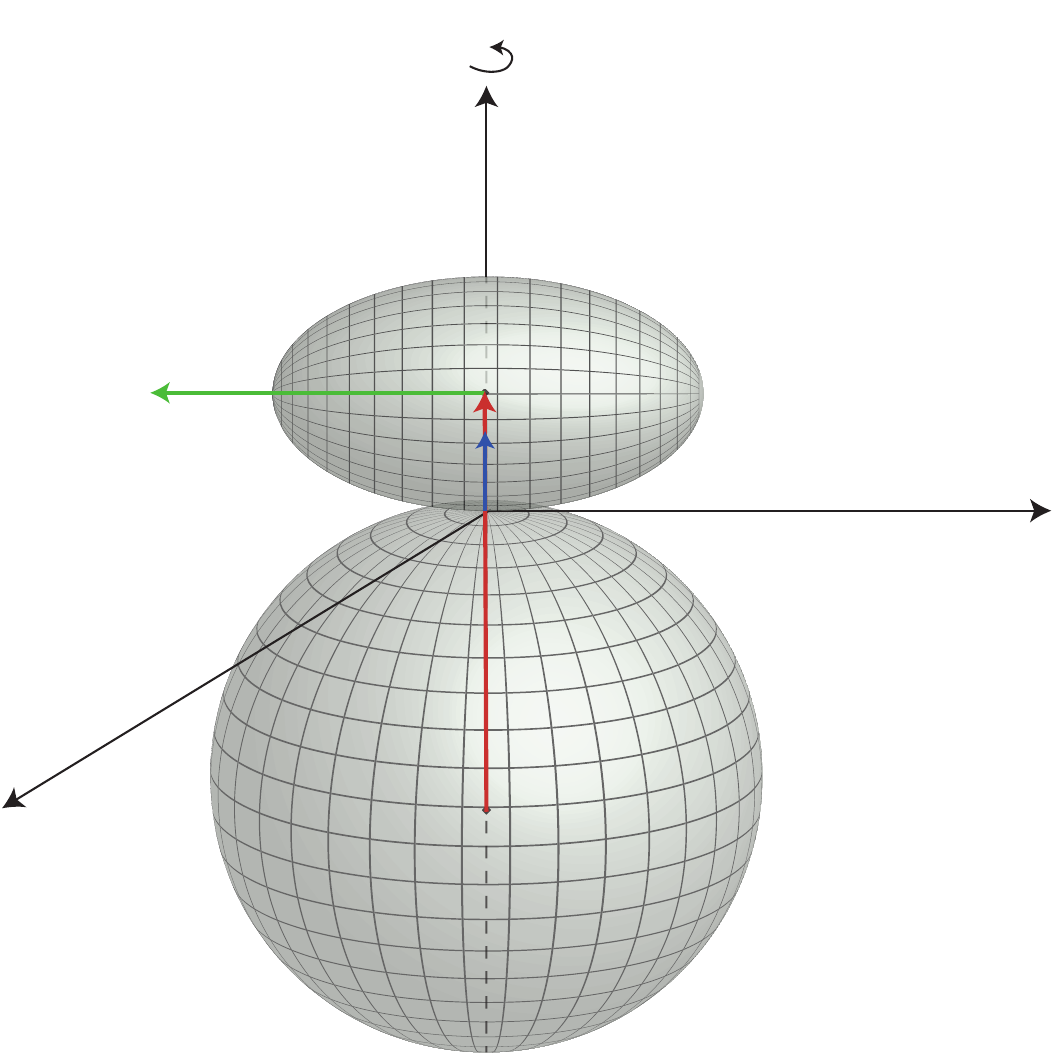}}; 
				\coordinate [label=below:\textcolor{black}{$\mathbf e_3$}] (E2) at (0.1,0.5);
				\coordinate [label=below:\textcolor{black}{$\mathbf q$}] (E2) at (0,1.2);
				\coordinate [label=below:\textcolor{black}{$\mathbf a$}] (E2) at (-2.5,1.3);
			\end{tikzpicture}
		\end{center}
		\caption{Piled-up bodies, case 1 (left) and case 2 (right). These are the only configurations where the bodies touch each other.}
  \label{fig:Piled-up1}
	\end{figure}
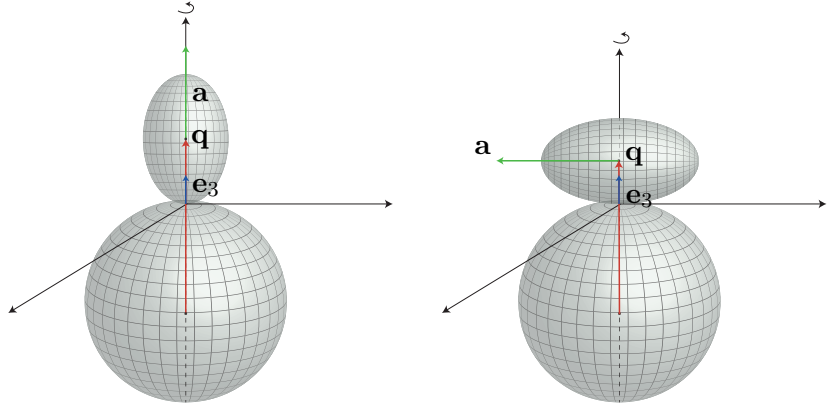
  \item[Case 6] In this case $\mathbf q ,\, \mathbf a $ and $ \mathbf e_z$ are parallel. Thus, the bodies are piled-up as in the left image in Figure~\ref{fig:Piled-up1}. For spherical bodies, these equilibria were first reported in \cite{Scheeres2006}. In order to get conditions for this arrangement of the bodies, we compute a basis of the columns null-space $\{\mathbf v_{\pmb\Xi}^1,\ldots,\mathbf v_{\pmb\Xi}^4 \}$ and $\pmb\varrho_{\pmb\Xi}$ as follows
  $$\mathbf v_{\pmb\Xi}^1=\mathbf e_4^6,\quad \mathbf v_{\pmb\Xi}^2=\mathbf e_3^6,\quad \mathbf v_{\pmb\Xi}^3=\mathbf e_2^6,\quad \mathbf v_{\pmb\Xi}^4=\mathbf e_1^6,\quad \pmb\varrho_{\pmb\Xi}=-\omega\,k_1\, w\, \mathbf e_6^6.$$
Step 3 does not provide any condition since the relations \eqref{eq:NecCond} are identically satisfied. Step 4 leads to the following relations 
$$a_{22} = -\frac{w^2 \left(  \omega \,k_1-\mu_1\right)}{2 z^2},\quad b_{21}= \frac{w \left( \omega\,k_1-\mu_1\right)}{z},$$
which encapsulate an intricate non-linear relation in the variables.
 
\end{description}

\section{Sufficient conditions for $N=2$, $m=1$}
\label{sec:Sufficient}

So far, we have been able to establish necessary conditions for the existence of relative equilibria. This section focuses on the case of a sphere and an axisymmetric body. The following result shows that, when the potential energy function is increasing with respect to the distance between the two bodies, the existence of relative equilibria can be ensured. In particular, this means that for any potential function derived from gravitational laws, as well as for any of their approximations, we can determine sufficient conditions leading to relative equilibria.

Before stating the main result of this section, we introduce the following notation. In the case $N=2$ and $m=1$, the potential function depends on the scalar quantities $R^2=\mathbf{q}\cdot \mathbf{q}$ and $R\cos\gamma=\mathbf{q}\cdot\mathbf{a}$, where $R$ denotes the magnitude of the position vector $\mathbf{q}$ and $\gamma$ is the angle between $\mathbf{q}$ and $\mathbf{a}$. Hence, the potential can be written as $V=V(R^2,R\cos\gamma)$.

Moreover, we use the notation
$$\dfrac{\partial V}{\partial R}=V_R=2R\, V_1+V_2 \cos\gamma,$$
where $V_i$ denotes the partial derivative of $V(x,y)$ with respect to its $i$-th argument.

\begin{theorem}
If $V_R>0$, then for any fixed values $R>0$ and $\gamma\in[0,\pi]$, the position and orientation vectors defined in \eqref{eq:RESuff} determine a relative equilibrium configuration, provided the following equation is satisfied and the angular velocity is set to the value given:
$$R^2b^3\,x-a\,u=\dfrac{c^4k_1}{R\,m_{12}\sin\gamma},\quad \omega^2=\dfrac{c^4}{R^2\,m_{12}\,(a+b\cos\gamma)}.$$
Then the solution is
\begin{equation}
    \label{eq:RESuff}
    \mathbf{q}=(x,0,z)=\dfrac{R}{c}(a+b\cos\gamma,0,b\sin\gamma),\quad 
\mathbf{a}=(u,0,w)=\dfrac{1}{c}(b+a\cos\gamma,0,-a\sin\gamma),
\end{equation}
where $a=2R\,V_1$, $b=V_2$, and $c^2=b^2+a^2-2ab\cos(\pi-\gamma)$.
\end{theorem}

\begin{proof}
Considering the definition given for $\mathbf{q}$ and $\mathbf{a}$, a straightforward computation shows that in deed for the given $\mathbf{q}$ and $\mathbf{a}$
we have
$\mathbf{q}\cdot\mathbf{q}=R^2$ and  $\mathbf{q}\cdot\mathbf{a}=R\cos\gamma$ as required.

\begin{figure}[h]
\centering
\begin{tikzpicture}[>=Stealth, line cap=round, line join=round, scale=1.0]
  \coordinate (O) at (0,0);
  \coordinate (A) at (3.0,2.4);
  \coordinate (B) at (5.0,0);
  \coordinate (D) at ($(A)+(1.5,1.2)$); 

  \draw[thick,->] (O) -- (6.0,0) node[right] {$x$};
  \draw[thick,->] (O) -- (0,3.4) node[above] {$z$};

  \draw[thick,->] (O) -- (A)
    node[midway, left=2pt] {$2V_{1}\,\mathbf{q}$};

  \draw[thick,->] (A) -- (B)
    node[midway, right=2pt] {$V_{2}\,\mathbf{a}$};

  \draw[very thick,->] (O) -- (B)
    node[midway, below=4pt] {$2V_{1}\,\mathbf{q}+V_{2}\,\mathbf{a}$};

  \draw[dashed,thick] (A) -- (D);

  \pic[draw,->,angle radius=10mm,angle eccentricity=1.25,
       "$\pi-\gamma$"] {angle = O--A--B};

  \pic[draw,->,angle radius=15mm,angle eccentricity=1.20,
       "$\gamma$"] {angle = B--A--D};

\end{tikzpicture}

\caption{Vector triangle illustrating $2V_{1}\,\mathbf{q}$, $V_{2}\,\mathbf{a}$, and their resultant.}
\label{fig:tri1}
\end{figure}
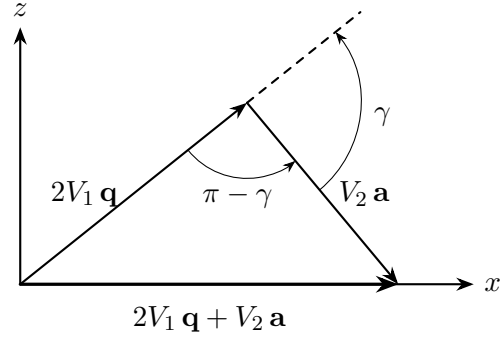

For the case $N=2$, $m=1$, and keeping in mind that $\mathbf{q}=(x,0,z)$ and $\mathbf{a}=(u,0,w)$, the equations of motion \eqref{eq:RE3aux} read as follows 
\begin{gather}
\begin{aligned}
    \label{eq:REN2m1}
    2V_1\,\mathbf{q}+V_2\,\mathbf{a}=&\,\omega^2\,m_{12}(x,0,0),\\
    V_2\,\mathbf{q}+\mu_1\,\mathbf{a}=&\,\omega^2\,k_{1}(0,0,1).
\end{aligned}
\end{gather}
Each of the above equations defines a triangle in the $xz$-plane; henceforth, we restrict our analysis to this plane.

The first equation defines the triangle depicted in Fig.~\ref{fig:tri1}. In order to ensure that this equation is satisfied, we must determine the value of $\omega$ such that the resultant of the vectors on the left-hand side of the equation equals the right-hand side, i.e.~ends on the $x$-axis. To do so, we consider the vertical component of the first equation in \eqref{eq:REN2m1}; we obtain
$$\dfrac{a}{R}z+b\,w=0.$$
Hence, assuming $z\neq0$, we have $a/R=bw/z$, which we plug into the equation for the first component of \eqref{eq:REN2m1}, leading to
$$b(u\,z-w\,x)=\omega^2m_{12}x\,z.$$
Noting that $(u\,z-w\,x)$ is the determinant of the projection of the vectors $c\,\mathbf{a}$ and $c/R\,\mathbf{q}$ in the $xz$-plane,
$$(u\,z-w\,x)=\det(c\,\mathbf{a},c/R\,\mathbf{q})=c^2\sin\gamma,$$
we obtain that $b\,c^2\sin\gamma=\omega^2m_{12}\,{R}/{c}(a+b\cos\gamma)\,{R}/{c}(b\sin\gamma)$, which simplifies to the condition for $\omega$ stated in the theorem.
Since $\omega^2$ is positive we did require that $V_r = a + b \cos\gamma > 0$.
\begin{figure}[htbp]
\centering
\begin{tikzpicture}[>=Stealth, line cap=round, line join=round, scale=1.1]

  \coordinate (O) at (0,0);
  \coordinate (A) at (0,3.2);
  \coordinate (B) at (2,1.6);
  \coordinate (D) at (4,0);

  \draw[thick,->] (O) -- (5.2,0) node[right] {$x$};
  \draw[thick,->] (O) -- (0,4.0) node[above] {$z$};

  \draw[thick,->] (O) -- (A)
    node[midway, left=2pt] {$\omega^2 k_1\mathbf{e}_3$};

  \draw[thick,->] (B) -- (A)
    node[midway, above right=2pt] {$\mu_1 \mathbf{a}$};

  \draw[thick,->] (O) -- (B)
    node[midway, below right=1pt] {$V_2 \mathbf{q}$};

  \pic[draw,->,angle radius=10mm,angle eccentricity=1.25,
       "$\gamma$"] {angle = A--B--O};

  \draw[dashed,thick] (B) -- (D);

\end{tikzpicture}
\caption{Vector triangle associated with the second equation in \eqref{eq:REN2m1}.}
\label{fig:tri2}
\end{figure}
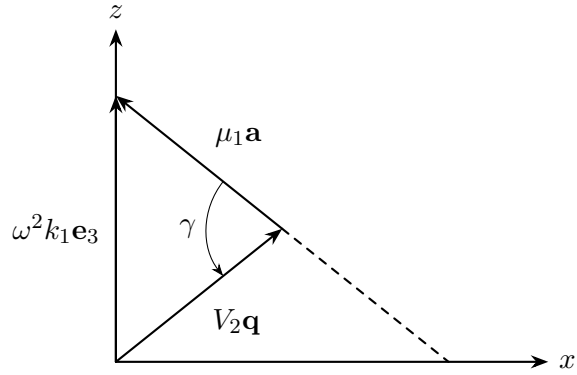

Another triangle depicted in Fig.~\ref{fig:tri2} and defined through the second equation in \eqref{eq:REN2m1} must be satisfied. This is done by proceeding in the same way than before with the equations for the horizontal and vertical components of the second equation in \eqref{eq:REN2m1}. In particular, from the horizontal component we obtain that
$$\mu_1=\dfrac{b+a\cos\gamma}{R\,b\,(a+b\cos\gamma)}.$$
Then, after some algebraic manipulations with the vertical component, we obtain the 
equation stated in the theorem. 

\end{proof}

\section*{Funding}
F. Crespo acknowledges support from internal start-up funds at Embry-Riddle Aeronautical University. Part of this work was carried out during a research visit to the Sydney Mathematical Research Institute, which provided financial support for the visit.

	\section*{Acknowledgment}
F.C. thanks the Sydney Mathematical Research Institute SMRI and especially Holger Dullin for their extraordinary hospitality and continued support during the research stay in the second semester of 2022, which allowed us to carry out this research. We thank E.A. Turner for his assistance with the figures in this work.
	
	\bibliographystyle{abbrv}
	\bibliography{Bibliography}
\end{document}